\documentclass[11pt]{article}
\usepackage[margin=1in]{geometry} 
\usepackage{amssymb,amsfonts,amsmath,bbm,mathrsfs,stmaryrd,mathtools,mathabx}
\usepackage{xcolor}
\usepackage{url}

\usepackage{enumerate}

\usepackage[colorlinks,
linkcolor=black!75!red,
citecolor=blue,
pdftitle={},
pdfproducer={pdfLaTeX},
pdfpagemode=None,
bookmarksopen=true,
bookmarksnumbered=true,
backref=page]{hyperref}
            
\usepackage{tikz}
\usetikzlibrary{arrows,calc,decorations.pathreplacing,decorations.markings,intersections,shapes.geometric,through,fit,shapes.symbols,positioning,decorations.pathmorphing}

\usepackage{braket}
\usepackage{tensor}

\usepackage{extarrows}

\usepackage[amsmath,thmmarks,hyperref]{ntheorem}
\usepackage{cleveref}

\creflabelformat{enumi}{#2(#1)#3}

\crefname{section}{Section}{Sections}
\crefformat{section}{#2Section~#1#3} 
\Crefformat{section}{#2Section~#1#3} 

\crefname{subsection}{\S}{\S\S}
\AtBeginDocument{%
  \crefformat{subsection}{#2\S#1#3}%
  \Crefformat{subsection}{#2\S#1#3}%
}

\crefname{subsubsection}{\S}{\S\S}
\AtBeginDocument{%
  \crefformat{subsubsection}{#2\S#1#3}%
  \Crefformat{subsubsection}{#2\S#1#3}%
}

\theoremstyle{plain}

\newtheorem{lemma}{Lemma}[section]
\newtheorem{proposition}[lemma]{Proposition}
\newtheorem{corollary}[lemma]{Corollary}
\newtheorem{theorem}[lemma]{Theorem}

\newtheorem{question}[lemma]{Question}

\theoremstyle{plain}
\theoremnumbering{Alph}
\newtheorem{theoremN}{Theorem}

\theoremstyle{plain}
\theorembodyfont{\upshape}
\theoremsymbol{\ensuremath{\blacklozenge}}

\newtheorem{definition}[lemma]{Definition}
\newtheorem{example}[lemma]{Example}

\newtheorem{remark}[lemma]{Remark}
\newtheorem{remarks}[lemma]{Remarks}

\newtheorem{notation}[lemma]{Notation}

\crefname{definition}{definition}{definitions}
\crefformat{definition}{#2definition~#1#3} 
\Crefformat{definition}{#2Definition~#1#3} 

\crefname{ex}{example}{examples}
\crefformat{example}{#2example~#1#3} 
\Crefformat{example}{#2Example~#1#3} 

\crefname{exs}{example}{examples}
\crefformat{examples}{#2example~#1#3} 
\Crefformat{examples}{#2Example~#1#3} 

\crefname{remark}{remark}{remarks}
\crefformat{remark}{#2remark~#1#3} 
\Crefformat{remark}{#2Remark~#1#3} 

\crefname{remarks}{remark}{remarks}
\crefformat{remarks}{#2remark~#1#3} 
\Crefformat{remarks}{#2Remark~#1#3} 

\crefname{convention}{convention}{conventions}
\crefformat{convention}{#2convention~#1#3} 
\Crefformat{convention}{#2Convention~#1#3} 

\crefname{notation}{notation}{notations}
\crefformat{notation}{#2notation~#1#3} 
\Crefformat{notation}{#2Notation~#1#3} 

\crefname{table}{table}{tables}
\crefformat{table}{#2table~#1#3} 
\Crefformat{table}{#2Table~#1#3}

\crefname{lemma}{lemma}{lemmas}
\crefformat{lemma}{#2lemma~#1#3} 
\Crefformat{lemma}{#2Lemma~#1#3} 

\crefname{proposition}{proposition}{propositions}
\crefformat{proposition}{#2proposition~#1#3} 
\Crefformat{proposition}{#2Proposition~#1#3} 

\crefname{corollary}{corollary}{corollaries}
\crefformat{corollary}{#2corollary~#1#3} 
\Crefformat{corollary}{#2Corollary~#1#3} 

\crefname{theorem}{theorem}{theorems}
\crefformat{theorem}{#2theorem~#1#3} 
\Crefformat{theorem}{#2Theorem~#1#3} 

\crefname{theoremN}{theorem}{theorems}
\crefformat{theoremN}{#2theorem~#1#3} 
\Crefformat{theoremN}{#2Theorem~#1#3} 

\crefname{enumi}{}{}
\crefformat{enumi}{(#2#1#3)}
\Crefformat{enumi}{(#2#1#3)}

\crefname{question}{question}{Questions}
\crefformat{question}{#2question~#1#3} 
\Crefformat{question}{#2Question~#1#3} 

\crefname{assumption}{assumption}{Assumptions}
\crefformat{assumption}{#2assumption~#1#3} 
\Crefformat{assumption}{#2Assumption~#1#3} 

\crefname{construction}{construction}{Constructions}
\crefformat{construction}{#2construction~#1#3} 
\Crefformat{construction}{#2Construction~#1#3} 

\crefname{equation}{}{}
\crefformat{equation}{(#2#1#3)} 
\Crefformat{equation}{(#2#1#3)}

\numberwithin{equation}{section}
\renewcommand{\theequation}{\thesection-\arabic{equation}}

\theoremstyle{nonumberplain}
\theoremsymbol{\ensuremath{\blacksquare}}

\newtheorem{proof}{Proof}
\newcommand\pf[1]{\newtheorem{#1}{Proof of \Cref{#1}}}

\newcommand\bC{{\mathbb C}}

\newcommand\bQ{{\mathbb Q}}
\newcommand\bR{{\mathbb R}}

\newcommand\bZ{{\mathbb Z}}

\newcommand\cC{{\mathcal C}}

\newcommand\cF{{\mathcal F}}

\newcommand\cP{{\mathcal P}}

\DeclareMathOperator{\id}{id}

\newcommand\numberthis{\addtocounter{equation}{1}\tag{\theequation}}

\newcommand{\cat}[1]{\textsc{#1}}

\newcommand\spr[1]{\cite[\href{https://stacks.math.columbia.edu/tag/#1}{Tag {#1}}]{stacks-project}}
\newcommand{\qedhere}{\mbox{}\hfill\ensuremath{\blacksquare}}

\title{Localized strict topologies on multiplier algebras of pro-$C^*$-algebras}
\author{Alexandru Chirvasitu}

\begin{document}

\date{}

\newcommand{\Addresses}{{
  \bigskip
  \footnotesize

  \textsc{Department of Mathematics, University at Buffalo}
  \par\nopagebreak
  \textsc{Buffalo, NY 14260-2900, USA}  
  \par\nopagebreak
  \textit{E-mail address}: \texttt{achirvas@buffalo.edu}


}}

\maketitle

\begin{abstract}
  The bounded localization $\beta_b$ of a locally convex topology $\beta$ is defined as the finest locally convex topology agreeing with $\beta$ on all bounded sets. We show that the strict topology on the multiplier algebra of a bornological pro-$C^*$-algebras equals its own localization, generalizing the analogous result due to Taylor for multiplier algebras of plain $C^*$-algebras.

  We also (a) characterize the barreled commutative unital pro-$C^*$-algebras as those of continuous functions on functionally Hausdorff spaces whose relatively pseudocompact subsets are relatively compact, equipped with the topology of uniform convergence on compact subsets, and (b) describe a contravariant equivalence between the category of commutative unital pro-$C^*$-algebras and a category of Tychonoff (rather than functionally Hausdorff) topological spaces.
\end{abstract}


\vspace{.5cm}


\tableofcontents

\section*{Introduction}

Consider the {\it multiplier algebra} $(M(A),\beta)$ \cite[Definition 2.2.2]{wo} of a $C^*$-algebra $A$, equipped with its {\it strict topology} $\beta$ \cite[Definition 2.3.1]{wo}. It is a non-obvious and occasionally convenient fact \cite[Corollary 2.7]{tay_strict} that $\beta$ equals its own {\it localization} with respect to bounded sets (in the terminology of \cite[discussion preceding \S IV.6.7]{schae_tvs}, say, or the precursor \cite{dorr_loc} to \cite{tay_strict}): the strongest locally convex topology $\beta_b$ on $M(A)$ agreeing with $\beta$ on the {\it bounded} \cite[\S I.5.1]{schae_tvs} subsets of $(M(A),\beta)$ is $\beta$ itself (as opposed to strictly stronger, as it could, in principle, have been). 

The present note was initially motivated by \cite{ct_ncclass_xv}, where it would have been useful, at one point, to have an analogue of Taylor's localization result for the {\it pro}-$C^*$-algebras \cite[Definition 1.1]{phil1}: {\it cofiltered limits} \spr{04AY} of $C^*$-algebras in the category of topological algebras.

Such a gadget ($A$, say) again has a multiplier algebra $M(A)$ \cite[Definition 3.13]{phil1} of {\it pro-}$C^*$-algebras \cite[Theorem 3.14]{phil1}
\begin{equation*}
  A\cong\varprojlim_{p\in S(A)} A_p,\quad \text{$C^*$-algebras }A_i,\quad A\to A_i\text{ surjective} 
\end{equation*}
(following \cite[discussion following Definition 1.1]{phil1} in denoting by $S(A)$ the set of continuous $C^*$-seminorms on $A$). Specifically, recall that \cite[Theorem 3.14]{phil1} exhibits two topologies of interest on $M(A)$:
\begin{itemize}
\item its own pro-$C^*$ topology $\tau$, resulting from the realization
  \begin{equation*}
    M(A)\cong \varprojlim_p M(A_p)
  \end{equation*}
  (with possibly non-surjective transition maps $M(A_i)\to M(A_j)$ \cite[example following Theorem 4.2]{apt});
\item and its strict topology $\beta$, as either
  \begin{itemize}
  \item the limit of the same inverse system with each $M(A_i)$ equipped with {\it its} strict topology;
  \item or, equivalently \cite[Theorem 3.14 (2)]{phil1}, the topology of pointwise convergence on $M(A)$ regarded as a set of operators
    \begin{equation*}
      A\ni a\xmapsto{\quad(l,r)\quad}(la,ra)\in A\oplus A,
    \end{equation*}
    as in \cite[Definition 3.13]{phil1} (with $A$ equipped with its pro-$C^*$ topology and $A\oplus A$ the usual \cite[\S II.4.5, Definition 2]{bourb_tvs} direct sum of locally convex topological vector spaces, or TVSs for short).
  \end{itemize}
\end{itemize}

Boundedness is definable \cite[Definition 14.1]{trev_tvs} for subsets of any TVS (see also \Cref{def:settypes} \Cref{item:bdd}), and the problem alluded to above is:

\begin{question}\label{qu:beta}
  Is the strict topology $\beta$ on the multiplier algebra $M(A)$ of a pro-$C^*$-algebra $A$ the finest locally convex topology $\beta_b$ agreeing with $\beta$ on $\tau$-bounded subsets of $M(A)$?
\end{question}

In this generality, the answer is `no' (\Cref{ex:taunetaub}). The issue remains, then, of singling sufficient conditions on a pro-$C^*$-algebra that will ensure localization does hold. \Cref{th:unifconvforfree} and \Cref{cor:sigmaok} do this, for a class broad enough to include {\it $\sigma$-$C^*$-algebras} (i.e. \cite[\S 5]{phil1} {\it countable} cofiltered limits of $C^*$-algebras). 

Recall (\cite[paragraph preceding Corollary 2 to Theorem 32.2]{trev_tvs} say, or \cite[\S III.2, Definition 1]{bourb_tvs}), for the purpose of stating (a summary of) the relevant result, that the {\it bornological} locally convex topological vector spaces $V$ are those for whom the continuity of linear maps defined thereon can be assessed based on their behavior on bounded sets: a linear map $V\xrightarrow{f}W$ to another locally convex TVS is continuous if and only if it sends bounded sets to bounded sets. The statement, then, is

\begin{theoremN}
  Let $A$ be a bornological pro-$C^*$-algebra (in particular, $A$ could be $\sigma$-$C^*$).

  The strict topology on the multiplier algebra $M(A)$ then equals its own localization with respect to the bounded subsets of $M(A)$.  \qedhere
\end{theoremN}

This line of inquiry, into boundedness in the locally convex topology of a pro-$C^*$-algebra, then naturally blends with examining such properties as being {\it barreled} \cite[Definition 33.1]{trev_tvs}: such that every {\it absolutely convex}, closed, {\it absorbent} subset is a neighborhood of the origin (the individual terms entering here are recalled in \Cref{def:settypes}). The property is essentially an abstraction away from the {\it Banach-Steinhaus (or uniform-boundedness)} \cite[Theorem 33.1]{trev_tvs}, and captures what one needs for that result to hold. It is also \cite[\S 28.1 (2)]{koeth_tvs-1}, for {\it complete} \cite[Definition 5.2]{trev_tvs} locally convex TVSs, somewhat weaker than being bornological.

Some partial information on the barreled-ness of {\it commutative} unital pro-$C^*$-algebras is provided in passing by \cite{inoue_loc}: while in general commutative unital $C^*$-algebras are \cite[Theorem 2.7 and Proposition 2.6]{phil1} spaces $C(X)$ of continuous complex-valued functions on a {\it functionally $T_2$} space $X$ (\Cref{res:whenbarr} \Cref{item:fcx}) equipped with the topology of uniform convergence on compact subsets of a sufficiently well-behaved ({\it distinguished} \cite[Definition 2.5]{phil1}) family $\cF\subseteq 2^X$ of compact subsets $F\subseteq X$, \cite[Remark 4.1]{inoue_loc} shows that being barreled implies $\cF$ must be {\it full} (i.e. contain all compact subsets).

All of this suggests a closer examination of commutative pro-$C^*$-algebras, the underlying topology of the corresponding spaces $X$, and how it relates to functional-analytic properties such as barreled-ness. \Cref{se:compl-comm-pro} pursues this train of thought in a number of directions; compressing and slightly paraphrasing \Cref{cor:alsot312} and \Cref{th:t312barrel}, for instance, we have 

\begin{theoremN}\label{th:introb}
  \begin{enumerate}[(1)]
  \item\label{item:introb-1} The commutative unital pro-$C^*$-algebras are precisely those of continuous functions on {\it Tychonoff} spaces $X$, equipped with the topology of uniform convergence on some {\it functionally distinguished} family $\cF\subseteq 2^X$ of compact subsets.  

  \item Such a pro-$C^*$-algebra is barreled precisely when $\cF$ is full (i.e. consists of all compact subsets) and every relatively pseudocompact subset of $X$ is relatively compact.  \qedhere
  \end{enumerate}
\end{theoremN}

Some elaboration:

\begin{itemize}
\item The {\it Tychonoff} spaces are also those referred to as $T_{3\frac 12}$ or $T_1$ and {\it completely regular} \cite[Definition 14.8]{wil_top}: points are closed, and for every point $x$ not belonging to a closed set $F$ there is a continuous function $X\xrightarrow{f}\bR$ with $f(x)=0$ and $f|_{F}\equiv 1$.

\item For {\it functionally distinguished} see \Cref{def:funcdist}. $\cF$ is required to contain singletons, be closed under taking subsets and finite unions, and, roughly speaking, determine which real-valued functions on the space are continuous.

\item That last condition is a weakening of the last item in \cite[Definition 2.5]{phil1} of a (plainly) distinguished family of compact sets.

\item Correspondingly, \Cref{th:introb} \Cref{item:introb-1} is a ``Tychonoff-flavored'' version of \cite[Theorem 2.7]{phil1}: it describes the category of commutative unital pro-$C^*$-algebras in terms of Tychonoff-spaces-and-families in place of completely-Hausdorff-spaces-and-families.

  This is on occasion useful and perhaps of some independent interest, e.g. for the purpose of appealing {\it directly} to the rich theory of measures \cite[II, Chapter 7]{bog_meas} on Tychonoff spaces. 
\end{itemize}

Other ancillary remarks are scattered throughout, such as
\begin{itemize}
\item a discussion of the tension between the properties of being compactly generated functionally $T_2$ and respectively $T_{3\frac 12}$ (\Cref{res:whenbarr} \Cref{item:cautioncomplhaus}), expanding slightly on that in \cite[\S 2]{phil1};

\item an adjunction between the respective categories of topological spaces in \Cref{pr:ktadj};

\item which later lifts to an equivalence (\Cref{le:ktextend} and \Cref{th:ktcorr}) when (functionally) distinguished families are super-added;

\item various natural examples of pathologies: e.g. full distinguished families do not always entail being barreled, so that the converse to \cite[Remark 4.1]{inoue_loc} does not hold (\Cref{res:whenbarr} \Cref{item:fcx} and \Cref{res:needextraconstr} \Cref{item:inoueagain}).
\end{itemize}

\subsection*{Acknowledgements}

This work is partially supported by NSF grant DMS-2001128.

I am grateful for many comments and pointers from M. Tobolski and A. Viselter. 

\section{Preliminaries}\label{se.prel}

The phrase `topological vector space' (complex unless specified otherwise) is conveniently rendered as `TVS'. The symbol `$\preceq$' placed between topologies means `is coarser than'; its opposite `$\succeq$', naturally, means `is finer than'.

We recall some common terminology relating to topological vector spaces $V$, assumed complex unless specified otherwise (virtually any textbook account will do for at least some version of each of the following terms).

\begin{definition}\label{def:settypes}
  A subset $K\subseteq V$ of a locally convex TVS is
  \begin{enumerate}  [(a)]

  \item  {\it convex} \cite[\S 15.10]{koeth_tvs-1} if it is closed under taking binary convex combinations
    \begin{equation*}
      K^2\ni (x,y)\xmapsto{\quad}\alpha x+(1-\alpha)y,\quad \alpha\in[0,1];
    \end{equation*}

  \item {\it balanced} \cite[Definition 3.2]{trev_tvs} or {\it circled} \cite[\S 15.1 (2)]{koeth_tvs-1} if it is closed under scaling by $\alpha$ with $|\alpha|\le 1$;

  \item {\it absolutely convex} \cite[\S 15.10]{koeth_tvs-1} if it is closed under linear combinations of the form
    \begin{equation*}
      K^2\ni (x,y)\xmapsto{\quad}\alpha x+\beta y,\quad |\alpha|+|\beta| \le 1.
    \end{equation*}
    Equivalently \cite[\S I.1, preceding Lemma 1]{rr_tvs}, absolutely convex = convex and balanced; 

  \item {\it absorbing} \cite[Definition 3.1]{trev_tvs} or {\it absorbent} \cite[\S 15.1 (2)]{koeth_tvs-1} if
    \begin{equation*}
      \forall x\in V,\quad \lambda x\in K\text{ whenever $|\lambda|$ is sufficiently small};
    \end{equation*}

  \item a {\it barrel} (\cite[Definition 7.1]{trev_tvs}, \cite[\S 21.2]{koeth_tvs-1}) if it is absolutely convex, absorbing and closed.

  \item\label{item:bdd} {\it bounded} (\cite[Definition 14.1]{trev_tvs}, \cite[\S 15.6]{koeth_tvs-1}) if it is {\it absorbed} by every origin neighborhood $U$:
    \begin{equation*}
      K\subseteq \lambda U
      \text{ if $|\lambda|$ is sufficiently large.}
    \end{equation*}
  \end{enumerate}  
\end{definition}


\section{Localized pro-$C^*$ topologies}\label{se:loc}

As explained above, the central theme here is \cite[Corollary 2.7]{tay_strict}, showing that the strict topology $\beta$ on the multiplier algebra $M(A)$ of a $C^*$-algebra $A$ equals its own localization on norm-bounded subsets of $M(A)$: the finest locally convex topology on $M(A)$ agreeing with $\beta$ on bounded subsets of $M(A)$ is $\beta$ itself.

\Cref{qu:beta} fits into a broader pattern: as noted in the proof of \cite[Corollary 2.3]{tay_strict}, for plain $C^*$-algebras $\beta$ and $\tau$ have the same bounded subsets (as a consequence of the uniform boundedness principle \cite[\S IV.2, Corollary 2]{rr_tvs}); we will see in \Cref{pr:multsamebdd} that the same holds for pro-$C^*$-algebras. For that reason, $\beta_b$ is the finest locally convex topology agreeing with $\beta$ on {\it its own} bounded sets; this is also what justifies the notation: the `b' subscript stands for `bounded'. The notational convention applies universally: for a locally convex topology $\sigma$ on a vector space we denote by $\sigma_b$ the finest locally convex topology agreeing with $\sigma$ on the latter's bounded subsets. One can thus also ask the obvious analogue:

\begin{question}\label{qu:tau}
  Is the pro-$C^*$-topology $\tau$ on the multiplier algebra $M(A)$ of a pro-$C^*$-algebra $A$ the finest locally convex topology $\tau_b$ agreeing with $\tau$ on the latter's bounded subsets?
\end{question}

\begin{remark}\label{re:tauisbeta}
  When $A$ is unital it is its own multiplier algebra and $\beta=\tau$; in that case, then, the two questions coalesce.
\end{remark}

We will see below (\Cref{ex:taunetaub}) that the questions generally have negative answers, but take the opportunity to examine cases where there are such topology coincidences.

\begin{remark}\label{re:tautaub}
  There is no reason, in general, why a locally convex topology $\tau$ would be equal to $\tau_b$. Consider, for instance, the {\it strong} and {\it ultra-strong} (or sometimes \cite[\S II.2]{tak1} {\it $\sigma$-strong}) topologies of \cite[\S I.3.1]{dixw}. 

  The two are distinct, with $\tau=$ strong topology strictly weaker, but agree on norm-bounded sets \cite[\S I.3.2]{dixw}. The uniform boundedness principle again shows that $\tau$ and the norm topology have the same bounded sets, so that $\tau_b$, being at least as strong as the ultra-strong topology, cannot coincide with $\tau$.

  Incidentally, this same example also shows that $\tau$ and $\tau_b$ need not even have the same continuous functionals \cite[\S I.3.3, Theorem 1]{dixw}. 
\end{remark}

We dispose first of one issue that might, conceivably, engender some confusion:

\begin{proposition}\label{pr:lcsamebdd}
  The bounded sets for a locally convex topology $\sigma$ are precisely those of $\sigma_b$.
\end{proposition}

We first need some notation and terminology.

\begin{definition}\label{def:gauge}
  Let $\cP$ be a basis \cite[Definition 7.7]{trev_tvs} of continuous seminorms on a locally convex topological vector space $(E,\sigma)$.
  \begin{enumerate}[(1)]

  \item A {\it gauge function} (or just plain {\it gauge}) for $\cP$ (or $E$, or $\sigma$) is a function $\gamma:\cP\to \bR_{>0}$. We might also use variants of this phrasing, such as {\it $E$- (or $\sigma$-)gauge}. 

  \item To each gauge $\gamma$ one can associate the {\it $\gamma$-ball}
    \begin{equation*}
      B_{\gamma}:=\{v\in E\ |\ p(v)\le \gamma(p(v)),\ \forall p\in \cP\}. 
    \end{equation*}

  \item To every function
    \begin{equation}\label{eq:gaugesel}
    \text{gauge }\gamma\xmapsto{\quad\nu\quad}\text{$\sigma$-$0$-neighborhood }V_{\gamma}\text{ in }E.
  \end{equation}
  we associate the absolute convex hull
  \begin{equation}\label{eq:wphi}
    W_{\nu}
    :=
    \mathrm{hull}\left(\bigcup_{\text{gauges }\gamma}B_{\gamma}\cap V_{\gamma}\right),
  \end{equation}
  of some use below. 
  \end{enumerate}
\end{definition}

\pf{pr:lcsamebdd}
\begin{pr:lcsamebdd}
  Because $\sigma_b$ is by definition stronger than $\sigma$, {\it its} bounded sets are certainly among those of $\sigma$. The implication of interest is thus
  \begin{equation*}
    \text{$\sigma$-bounded}
    \Longrightarrow
    \text{$\sigma_b$-bounded}.
  \end{equation*}
  Suppose $\sigma$ is defined \cite[Proposition 7.6]{trev_tvs} by a family of seminorms $p$ on the underlying vector space $E$. The absolute convex hulls \Cref{eq:wphi} form a basis of $\sigma_b$-0-neighborhoods as $\nu$ ranges over functions \Cref{eq:gaugesel}. Indeed:
  \begin{itemize}
  \item every absolutely convex $\sigma_b$-neighborhood of the origin contains some $W_{\nu}$ by definition;
  \item while the intersection of $W_{\nu}$ with any gauge ball $B_{\gamma_0}$ will be
    \begin{equation*}
      B_{\gamma_0}\bigcap \mathrm{hull}\left(V_{\gamma_0}\bigcup\text{some other set}\right),
    \end{equation*}
    with the right-hand absolutely convex hull being an open $\sigma$-origin-neighborhood.
  \end{itemize}  
  It is thus enough to argue that, having fixed the function $\nu$ and a $\sigma$-bounded set $B$, some scalar multiple of $B$ is contained in $W_{\nu}$. This, though, is immediate: $B$ itself, scaled, is contained in some $B_{\gamma}$. The $\sigma$-0-neighborhood
  \begin{equation*}
    V_{\gamma} = \nu\left(\gamma\right)
  \end{equation*}
  is then a 1-ball attached to some continuous seminorm $p$, and $B$ can always be scaled further to ensure that
  \begin{equation*}
    p(v)\le 1,\ \forall v\in B\text{ i.e. }B\subseteq V_{\gamma}
  \end{equation*}
  in addition to $B\subseteq B_{\gamma}$.
\end{pr:lcsamebdd}

Recall \cite[\S 28.1]{koeth_tvs-1} that a locally convex TVS is {\it bornological} if absolutely convex sets absorbing every bounded set are origin neighborhoods. Because every locally convex TVS has attached to it a {\it finest} locally convex topology admitting the same bounded sets and that topology is automatically bornological \cite[\S 28.2]{koeth_tvs-1}, \Cref{pr:lcsamebdd} (and the fact that metrizability entails being bornological \cite[\S 28.1 (4)]{koeth_tvs-1}) immediately implies

\begin{corollary}\label{cor:bornok}
  For a bornological (in particular, metrizable) locally convex TVS $(E,\sigma)$ we have $\sigma=\sigma_b$.  \qedhere
\end{corollary}

\begin{remark}\label{re:bornmack}
  Recall that the {\it Mackey topology} \cite[\S 21.4]{koeth_tvs-1} on a locally convex topological vector space $E$ is the topology of uniform convergence on absolutely convex weak$^*$-compact subsets of the continuous dual $E^*$. It is also the finest locally convex topology on $E$ that gives the same continuous dual as the original topology (the celebrated {\it Mackey-Arens theorem} \cite[\S 21.4 (2)]{koeth_tvs-1}).
  
  Every bornological space is {\it Mackey}, in the sense that its original topology coincides with its Mackey topology; this follows, say, from the definition of a {\it quasi-barreled space} in \cite[\S 23.4]{koeth_tvs-1} as a space equipped with a certain topology that is always \cite[\S 21.5 (2)]{koeth_tvs-1} than the Mackey topology, together with the fact \cite[\S 28.1 (1)]{koeth_tvs-1} that being bornological implies being quasi-barreled. 
\end{remark}

Also in the spirit of \Cref{pr:lcsamebdd}:

\begin{proposition}\label{pr:multsamebdd}
  For a pro-$C^*$-algebra $A$, the subsets of $M(A)$ bounded for the topologies $\beta$ (strict), $\tau$ (pro-$C^*$), $\beta_b$ and $\tau_b$ all coincide.
\end{proposition}
\begin{proof}
  Once we have shown that $\beta$ and $\tau$ have the same bounded sets the rest follows from \Cref{pr:lcsamebdd}, so we focus on that first claim.

  Again, $\tau$ being stronger than $\beta$, only the implication
  \begin{equation*}
    \text{$\beta$-bounded}
    \Longrightarrow
    \text{$\tau$-bounded}
  \end{equation*}
  is interesting. The ordinary $C^*$ version is recorded in \cite[proof of Corollary 2.3]{tay_strict}. In general, write
  \begin{equation*}
    A\cong \varprojlim_p A_p    
  \end{equation*}
  as an inverse limit with surjective $\pi_p:A\to A_p$ (as we always can \cite[Corollary 1.12]{phil1}), so that \cite[Theorem 3.14]{phil1} applies. A set $S\subset M(A)$ is the $\beta$-bounded if and only if its images
  \begin{equation*}
    \pi_p(S)\subset M(A_p) 
  \end{equation*}
  are all $\beta$-bounded by \cite[\S V.4, Proposition 14]{rr_tvs} (where we have extended $\pi_p$ to a map $M(A)\to M(A_p)$ \cite[Proposition 3.15 (1)]{phil1}). In turn, by the already-mentioned ordinary-$C^*$ remark, this is equivalent to all $\pi_p(S)$ being norm-bounded. But then another application of \cite[\S V.4, Proposition 14]{rr_tvs} equates this to $S$ being pro-$C^*$-bounded.
\end{proof}

We will thus speak of bounded subsets of $M(A)$ completely unambiguously: all topologies under consideration impose the same boundedness constraint.

The arguments in \cite{tay_strict} sometimes make explicit use of factorization theorems usually proved only in the more familiar Banach-space setup of the more familiar plain (as opposed to pro-)$C^*$-algebra theory, but those aspects of the argument can occasionally be dispensed with by reducing the pro-$C^*$ version back to the ordinary $C^*$ setup. A case in point is the following analogue of \cite[Corollary 2.3 and part of Corollary 2.7]{tay_strict}.

We will need the fact that a pro-$C^*$-algebra $A$ has an {\it approximate unit} (or {\it identity}) \cite[Corollary 3.12]{phil1}: an increasing net $(e_{\lambda})_{\lambda}$ of positive elements with $p(e_{\lambda})\le 1$ for all continuous $C^*$-seminorms $p\in S(A)$ with
\begin{equation*}
  ae_{\lambda},\ e_{\lambda}a\xrightarrow[\lambda]{\quad}a,\quad \forall a\in A.
\end{equation*}


Recall \cite[Example IV following Proposition 19.2]{trev_tvs} that the {\it strong dual} topology on the continuous dual of a topological vector space is the topology of uniform convergence on bounded sets. 

\begin{theorem}\label{th:bbsamefunc}
  Let $A$ be a pro-$C^*$-algebra and $M(A)$ its multiplier algebra.

  \begin{enumerate}[(1)]
  \item\label{item:embdiamond} The restriction $M(A)^*\to A^*$ then induces embeddings
    \begin{equation}\label{eq:embdiamond}
      \begin{tikzpicture}[auto,baseline=(current  bounding  box.center)]
        \path[anchor=base] 
        (0,0) node (l) {$M(A)_{\beta}^*$}
        +(3,.5) node (u) {$M(A)_{\beta_b}^*$}
        +(3,-.5) node (d) {$A_{\tau}^*$}
        +(6,0) node (r) {$A_{\tau_b}^*$,}
        ;
        \draw[right hook->] (l) to[bend left=6] node[pos=.5,auto] {$\scriptstyle $} (u);
        \draw[right hook->] (u) to[bend left=6] node[pos=.5,auto] {$\scriptstyle $} (r);
        \draw[->] (l) to[bend right=6] node[pos=.5,auto,swap] {$\scriptstyle \cong$} (d);
        \draw[right hook->] (d) to[bend right=6] node[pos=.5,auto,swap] {$\scriptstyle $} (r);
      \end{tikzpicture}
    \end{equation}
    homeomorphic onto their respective images when the spaces are all equipped with their strong dual topologies. 

  \item\label{item:t=tb} If furthermore $\tau$ and $\tau_b$ have the same continuous functionals, so that the lower right-hand arrow is an identification, so are the two top arrows.
  \end{enumerate}
\end{theorem}
\begin{proof}
  \begin{enumerate}[]
  \item {\bf \Cref{item:embdiamond}} The north-eastward embeddings follow from the strength ordering of the topologies:
    \begin{equation*}
      \beta\preceq\beta_b
      \quad\text{and}\quad
      \tau\preceq\tau_b
    \end{equation*}
    On the other hand, the south-eastward arrows are embeddings because $A$ is $\beta_b$- (hence also $\beta$-)dense in $M(A)$: because approximate identities $(e_{\lambda})$ consist (by \cite[Definition 3.10]{phil1}) of uniformly bounded elements in the sense that
    \begin{equation*}
      p(e_{\lambda})\le 1,\ \forall \lambda,\ \forall p\in S(A),
    \end{equation*}
    the density argument used to prove \cite[Theorem 3.14 (6)]{phil1} works just as well for $\beta_b$.

    The strong topology on either $M(A)_{\beta}^*$ or $M(A)_{\beta_b}^*$ equals that inherited from the strong topology on $M(A)^*_{\tau}$ and $M(A)^*_{\tau_b}$ respectively, because $\beta$, $\beta_b$, $\tau$ and $\tau_b$ all have the same bounded sets (\Cref{pr:multsamebdd}). That the embeddings are also homeomorphisms onto their images amounts to this (focusing on the plain rather than localized topologies; the argument goes through all the same): if
    \begin{equation*}
      f_{\alpha}\xrightarrow[\quad\alpha\quad]{} 0,\quad f_{\alpha}\in M(A)_{\beta}^*
    \end{equation*}
    uniformly on every bounded subset of $A$, then this happens uniformly on bounded subsets $B\subset M(A)$. 

    This is where the approximate unit $(e_{\lambda})_{\lambda\in \Lambda}$ will come in handy again: the assumption implies that for large $\alpha$ the functionals $f_{\alpha}$ will be small on the bounded subset
    \begin{equation*}
      \bigcup_{\lambda}Be_{\lambda} = \left\{xe_{\lambda}\ |\ x\in B,\ \lambda\in \Lambda\right\}\subset A,
    \end{equation*}
    hence also on $B$ itself by the assumed strict continuity of $f_{\alpha}$:

    \begin{equation*}
      xe_{\lambda}\xrightarrow[\lambda]{\quad}x\text{ strictly}
      \quad
      \Longrightarrow
      \quad
      f_{\alpha}(xe_{\lambda})\xrightarrow{\quad} f_{\alpha}(x). 
    \end{equation*}
    As for the bijectivity of the bottom left-hand arrow in \Cref{eq:embdiamond}:
    
    Every element of $M(A)_{\tau}^*$ factors through some $M(A)\to M(A_p)$ because $\tau$ is by definition induced by the norm topologies on $M(A_p)$ and every functional is continuous with respect to one of those seminorms \cite[Proposition 7.7]{trev_tvs}. The same goes for the functionals in $A_{\tau}^*$ (each factors through some $A\to A_p$) as well as those in $M(A)_{\beta}^*$ (since $\beta$ is the inverse limit of the strict topologies on $M(A_p)$ \cite[Theorem 3.14 (2)]{phil1}). The restriction $M(A)_{\beta}^*\to A_{\tau}^*$ can thus be recovered as a limit of restrictions
    \begin{equation*}
      M(A_p)^*\xrightarrow{\quad} M(A_p)_{\beta}^*\cong A_p,\quad p\in S(A),
    \end{equation*}
    the isomorphism being \cite[Corollary 2.3]{tay_strict} (because $A_p$ is a $C^*$-algebra).
    
  \item {\bf \Cref{item:t=tb}} This follows immediately from part \Cref{item:embdiamond}: under the additional hypothesis the lower arrows in \Cref{eq:embdiamond} are both isomorphisms (of topological vector spaces), hence so are the top embeddings.
  \end{enumerate}
\end{proof}

The hypothesis of \Cref{th:bbsamefunc} \Cref{item:t=tb} need not hold, even for unital $A$ (in which case $M(A)=A$ and $\tau=\beta$, so that the {\it conclusion} fails to hold as well):

\begin{example}\label{ex:taunetaub}
  Consider the pro-$C^*$-algebra $A$ of \cite[Example 2.11]{phil1}, or any number of analogues: it is the usual $C(X)$ for some compact metric space $X$, with the pro-$C^*$ topology induced by the countable closed subsets $F\subseteq X$ with finitely many cluster points (i.e. the topology of uniform convergence on each such set). $A$ being unital in this case, it equals its multiplier algebra and the strict topology specializes back to the pro-$C^*$ topology $\tau$.

  As soon as $X$ is uncountable, $\tau$ and $\tau_b$ will not have the same continuous functionals. To see this, consider subsets $F_n\subset X$, $n\in \bZ_{>0}$, each consisting of a sequence and its limit $\ell_n$, so that $\ell_n\to \ell\in X$; then set
  \begin{equation*}
    K:=\bigcup_{n}F_n\cup\{\ell\}
  \end{equation*}
  and consider a probability measure $\mu$ supported on $K$ and assigning mass $\frac 1{2^n}$ to $F_n$. As a functional on $C(X)$, $\mu$ is easily seen to be $\tau_b$-continuous on every bounded subset (here, boundedness specializes back to the familiar notion of uniform boundedness on $X$); nevertheless, $\mu\in C(X)_{\tau_b}^*$ will not be $\tau$-continuous because $\mu$ is not supported on any finite-cluster $F$.
\end{example}


\begin{remark}\label{re:bornprocastok}

  {\it Bornological} pro-$C^*$-algebras are better behaved, in that \Cref{cor:bornok} shows that in that case $\tau=\tau_b$. In particular, this applies to {\it $\sigma$-$C^*$-algebras} \cite[\S 5]{phil1}, i.e. pro-$C^*$-algebras whose topology is induced by countably many $C^*$-seminorms: since $\sigma$-$C^*$-algebras are metrizable \cite[\S 18.2 (2)]{koeth_tvs-1}, they are bornological \cite[\S 28.1 (4)]{koeth_tvs-1}.
\end{remark}

The characterization of $\beta$-equicontinuous subsets of $M(A)_{\beta}^*$ of \cite[Theorem 2.6]{tay_strict} also goes through, provided we interpret its condition (2) appropriately: where that result requires uniform boundedness (which makes sense for functionals on a Banach space, such as $M(A)$ for ordinary $C^*$-algebras $A$) we require $\tau$-equicontinuity:

\begin{theorem}\label{th:equicontonma}
  Let $A$ be a pro-$C^*$-algebra with an approximate unit $(e_{\lambda})$ and $M(A)$ its multiplier algebra. For a set $S\subset M(A)_{\beta}^*$ of strictly-continuous functionals, the following conditions are equivalent:
  \begin{enumerate}[(a)]
  \item\label{item:bequi} $S$ is $\beta$-equicontinuous.
  \item\label{item:tequi} $S$ is $\tau$-equicontinuous and
    \begin{equation}\label{eq:flambda0}
      (1-e_{\lambda})f(1-e_{\lambda})
      \xrightarrow[\quad\lambda\quad]{}0
      \text{ strongly and uniformly in }f\in S. 
    \end{equation}
  \end{enumerate}
\end{theorem}
\begin{proof}
  Strictly-continuous functionals can be regarded as functionals on $A$ by \Cref{th:bbsamefunc}, and we do so regard them. Whether one assumes $\tau$- or $\beta$-equicontinuity, $S$ will vanish on $\ker p$ for some continuous $C^*$-seminorm $p\in S(A)$, and the statement reduces to its plain $C^*$ counterpart \cite[Theorem 2.6]{tay_strict}.
\end{proof}

\begin{remark}\label{re:conway}
  The commutative precursor to the characterization of strict equicontinuity in \cite[Theorem 2.6]{tay_strict} is \cite[Theorem 2]{conw_strict-1}, on spaces of measures on a locally compact topological space.
\end{remark}

A different choice for transporting \cite[Theorem 2.6, condition (2)]{tay_strict} to the present setting would have been to require that $S$ be {\it bounded} in the strong topology (i.e. that $S$ be uniformly bounded on every $\tau$-bounded subset). The implication \Cref{item:bequi} $\Longrightarrow$ \Cref{item:tequi} of \Cref{th:equicontonma} still holds then, since $\tau$-equicontinuity certainly implies strong boundedness. The converse does {\it not} hold though, even in the unital case.

\begin{example}\label{ex:fincluster}
  The selfsame \Cref{ex:taunetaub} will do, again with uncountable $X$. Consider the family $S\subset A^*$ of all probability measures concentrated on sets $F$ as above. That family is
  \begin{itemize}
  \item bounded on every $\tau$-bounded subset of $A$, since such subsets are easily seen to be bounded in the usual uniform topology on $A=C(X)$;
  \item but not equicontinuous, since there is no single countable, finite-cluster $F$ on which all elements of $S$ are concentrated.
  \end{itemize}
\end{example}

We will see shortly that the second condition of \Cref{th:equicontonma} \Cref{item:tequi} essentially ``comes along for the ride'' in all cases of interest. Moreover, unwinding the proof of \cite[Corollary 2.7]{tay_strict} (which in turn follows the general plan of \cite[Theorem]{dorr_loc}), gives a handy criterion for the equality $\beta=\beta_b$ of topologies on $M(A)$.

Recall the notion of {\it polar} \cite[Definition 19.1]{trev_tvs} of a set $V\subset E$ of a topological vector space:
\begin{equation*}
  V^{\circ}:=\{f\in E^*\ |\ |f(x)|\le 1,\ \forall x\in V\}.
\end{equation*}

\begin{theorem}\label{th:unifconvforfree}
  Let $A$ be a pro-$C^*$ algebra with multiplier algebra $M(A)$.

  \begin{enumerate}[(1)]
  \item\label{item:allconv} For every function \Cref{eq:gaugesel} the polar
    \begin{equation*}
      W_{\nu}^{\circ} = \bigcap_{\text{gauges }\gamma}\left(B_{\gamma}\cap V_{\gamma}\right)^{\circ} \subset M(A)_{\beta_b}^*
    \end{equation*}
    for $W_{\nu}$ as in \Cref{eq:wphi} satisfies \Cref{eq:flambda0}. 

  \item\label{item:ifftauequicont} If $\tau$ and $\tau_b$ have the same continuous functionals, then the following conditions are equivalent:
    \begin{enumerate}[{(\ref{item:ifftauequicont}}a{)}]
    \item\label{item:bb} $\beta=\beta_b$. 
    \item\label{item:wtauequic} Every $W_{\nu}^{\circ}$ is $\tau$-equicontinuous.
    \item\label{item:wfactors} Every $W_{\nu}^{\circ}$ consists of functionals factoring through $A\to A_p$ (hence also $M(A)\to M(A_p)$) for some $p\in S(A)$. 
    \end{enumerate}    
  \end{enumerate}
\end{theorem}
\begin{proof}
  \begin{enumerate}[]
  \item {\bf \Cref{item:ifftauequicont}, assuming \Cref{item:allconv}} Given the assumption that $\tau$ and $\tau_b$ have the same continuous functionals, we know from \Cref{th:bbsamefunc} that $M(A)_{\beta}^*$ and $M(A)_{\beta_b}^*$ are equal as subspaces of $M(A)^*$.

    The $\beta_b$-closed absolute convex hulls
    \begin{equation*}
      \overline{W_{\nu}}^{\beta_b}
      =
      \overline{\mathrm{hull}\left(\bigcup_{\gamma}B_{\gamma}\cap V_{\gamma}\right)}^{\beta_b}
    \end{equation*} 
    form a fundamental system of closed $\beta_b$-neighborhoods of $0\in M(A)$ (hence also a fundamental system of neighborhoods \cite[\S 18.1 (3)]{koeth_tvs-1}), so \Cref{item:bb} is equivalent to the claim that each such hull contains some $\beta$-0-neighborhood
    \begin{equation}\label{eq:vpa}
      V_{p\mid (a_i)} = V_{p\mid a_1,\cdots,a_n}:=\{x\in M(A)\ |\ p(xa_i)\le 1,\ p(a_i x)\le 1\}
    \end{equation}
    for a continuous $C^*$-seminorm $p\in S(A)$ and $a_i\in A$. $W_{\nu}$ being $\beta_b$-closed and absolutely convex, it is its own double polar (the {\it bipolar} of \cite[Definition 35.1]{trev_tvs}) by \cite[Proposition 35.3]{trev_tvs}. We thus have 
    \begin{align*}
      W_{\nu}=W_{\nu}^{\circ\circ}\supseteq V_{p\mid (a_i)}
      &\xLeftrightarrow{\quad}\numberthis\label{eq:wcirccirc}
      W_{\nu}^{\circ}\subseteq V_{p\mid (a_i)}^{\circ}\\
      &\xLeftrightarrow{\quad\text{\cite[Proposition 32.7]{trev_tvs}}\quad}
      W_{\nu}^{\circ}\text{ is $\beta$-equicontinuous}.
    \end{align*}
    Because we are assuming we have already proven part \Cref{item:allconv}, \Cref{th:equicontonma} equates this to the $\tau$-equicontinuity of $W_{\nu}^{\circ}$; we thus indeed have $\text{\Cref{item:bb}}\xLeftrightarrow{\quad}\text{\Cref{item:wtauequic}}$.

    It is obvious that \Cref{item:bb} (via \Cref{eq:wcirccirc}) implies \Cref{item:wfactors}, since of course all elements of 
    \begin{equation*}
      V_{p\mid (a_i)}^{\circ}\subset M(A)_{\beta}^*
    \end{equation*}
    factor through $M(A)\to M(A_p)$. Conversely, if $W_{\nu}^{\circ}$ factors through $M(A)\to M(A_p)$ then \Cref{eq:wcirccirc} holds (for some neighborhood $V_{p\mid (a_i)}$) because we already know \cite[Corollary 2.7]{tay_strict} that $\beta=\beta_b$ on $M(A_p)$.

  \item {\bf \Cref{item:allconv}} Unpacked, the claim is that for every $A$-gauge $\gamma_0$ we have
    \begin{equation*}
      f\left((1-e_{\lambda})x(1-e_{\lambda})\right)\xrightarrow[\quad\lambda\quad]{}0
    \end{equation*}
    uniformly in $x\in B_{\gamma_0}$ and $f\in W_{\nu}^{\circ}$. The argument transports over from \cite[Corollary 2.7]{tay_strict} with only minor changes. We need
    \begin{equation}\label{eq:fsmall}
      \left|f\left(C(1-e_{\lambda})x(1-e_{\lambda})\right)\right|\le 1,\ \forall x\in B_{\gamma_0},\ \forall f\in W_{\nu}^{\circ}
    \end{equation}
    for sufficiently large $\lambda$ (henceforth $\lambda \gg 0$) and fixed but arbitrary (arbitrarily large, that is) $C\in \bR_{>0}$. The neighborhood $V_{C\gamma_0} = \nu(C\gamma_0)$ attached by $\nu$ to the scaled gauge $C\gamma_0$ contains some neighborhood of the form \Cref{eq:vpa}, and for $\lambda\gg 0$ we have
    \begin{equation*}
      p\left((1-e_{\lambda})a_i\right),\ 
      p\left(a_i(1-e_{\lambda})\right)
      \le \frac{1}{C\cdot \gamma_0(p)},\ \forall i.
    \end{equation*}
    This means that
    \begin{equation*}
      C(1-e_{\lambda})x(1-e_{\lambda})\in B_{C\gamma_0}\cap V_{p\mid (a_i)},
    \end{equation*}
    hence \Cref{eq:fsmall} by
    \begin{equation*}
      f\in W_{\nu}^{\circ}\subseteq \left(B_{C\gamma_0}\cap V_{p\mid (a_i)}\right)^{\circ}. 
    \end{equation*}
  \end{enumerate}
  This completes the proof. 
\end{proof}

\begin{remark}\label{re:dorroh-no-closure}
  \cite[Theorem]{dorr_loc} and \cite[Corollary 2.7]{tay_strict} follow the same strategy in showing that $\beta=\beta_b$ (the former for {\it commutative} $C^*$-algebras, the latter for arbitrary ones): starting with a closed absolutely convex $\beta_b$-neighborhood $W$ of $0\in M(A)$, one produces a $\beta_b$-closed absolutely convex subset $W'\subseteq W$ (playing the same role as $W_{\nu}$ in the proof of \Cref{th:unifconvforfree}) whose polar $(W')^{\circ}$ is $\beta$-equicontinuous.


  At this point, there is a small subtlety that we found somewhat puzzling on a first perusal of these sources: while the proof of \cite[Theorem]{dorr_loc} concludes from this that $W'$ itself must be a $\beta$-neighborhood of 0, the proof of \cite[Corollary 2.7]{tay_strict} further takes the $\beta$-closure of $W'$. This is not necessary, as $W'$ is already $\beta$-closed: since at that stage in the proof one already knows that $\beta$ and $\beta_b$ have the same continuous functionals, their respective closed {\it convex} sets also coincide \cite[Proposition 18.3]{trev_tvs}.
\end{remark}

As an immediate consequence of \Cref{th:unifconvforfree} we now have

\begin{corollary}\label{cor:sigmaok}
  For a bornological pro-$C^*$-algebra $A$ (so in particular for $\sigma$-$C^*$-algebras), the topologies $\beta$ and $\beta_b$ on its multiplier algebra $M(A)$ coincide. 
\end{corollary}
\begin{proof}
  We check that the requirements of \Cref{th:unifconvforfree} \Cref{item:ifftauequicont} are met. \Cref{re:bornprocastok} already confirms that $\tau=\tau_b$, so in particular these two have the same continuous functionals.
  
  Being bornological and complete, $A$ is also barreled \cite[\S 28.1 (2)]{koeth_tvs-1}. For barreled spaces pointwise-bounded sets of linear maps into locally convex topological vector spaces are equicontinuous \cite[Theorem 33.1]{trev_tvs}; since the sets $W_{\nu}^{\circ}$ of \Cref{th:unifconvforfree} \Cref{item:ifftauequicont} are plainly pointwise-bounded they must be $\tau$-equicontinuous, and we are done.
\end{proof}

\begin{remarks}\label{res:whenbarr}

  \begin{enumerate}[(1)]

  \item\label{item:fcx} A commutative unital pro-$C^*$-algebra is \cite[Proposition 2.6 and Theorem 2.7]{phil1} uniquely of the form
    \begin{equation}\label{eq:fc}
      \tensor[_{\cF}]{C(X)}{}:=\left\{
        \text{
          \begin{minipage}{10cm}
            continuous $\bC$-valued functions on $X$ with the topology of uniform convergence on members of $\cF$
          \end{minipage}
        }
      \right\},
    \end{equation}
    where
    \begin{itemize}
    \item $X$ is a {\it completely (or functionally) Hausdorff (or $T_2$)} \cite[Definition 2.2]{phil1} topological space in the sense that any two distinct points are distinguishable by some continuous real-valued function on $X$.

    \item And $\cF$ is a {\it distinguished family} \cite[Definition 2.5]{phil1} of compact subsets of $X$:
      \begin{itemize}
      \item containing the singletons;

      \item closed under taking compact subsets;

      \item closed under forming finite (or equivalently, binary) unions;
        
      \item and determining the topology (so that the space is {\it compactly generated}, or a {\it $\kappa$-space} \cite[Definition 43.8]{wil_top}):
        \begin{equation}\label{eq:topdet}
          A\subseteq X\text{ closed }
          \iff
          A\cap F\subseteq F\text{ closed },\ \forall F\in \cF.
        \end{equation}
      \end{itemize}
    \end{itemize}
    \cite[Remark 4.1]{inoue_loc} (paraphrased) observes that if a unital commutative pro-$C^*$-algebra $A\cong \tensor[_{\cF}]{C(X)}{}$ is barreled then $\cF$ consists of on {\it all} compact subsets of $X$. It is worth noting, though, that this is only a {\it necessary} condition for being barreled:
    \begin{equation*}
      X:=\text{first uncountable ordinal $\Omega$}
    \end{equation*}
    \cite[Remark following Proposition 12.2]{mich_loc} equipped with the {\it order topology} \cite[\S 14]{mnk} violates the converse, because it satisfies the hypotheses of \Cref{pr:whichbarr}. Indeed, every continuous real-valued function on $\Omega$ is bounded: this is noted obliquely in the already-cited \cite[Remark after Proposition 12.2]{mich_loc}; alternatively, it follows \cite[Theorem 17.13]{wil_top} from the fact that $\Omega$ is {\it countably compact} \cite[Definition 17.1]{wil_top}: every {\it countable} open cover has a finite subcover (see \cite[Example 17.2 (c)]{wil_top}, where our $\Omega$ appears as $\mathbf{\Omega_0}$).
    
  \item\label{item:cautioncomplhaus} Some caution is in order regarding the term `completely Hausdorff', whose meaning varies by source (hence the preference for `functionally $T_2$' in the sequel). 

    In \cite[Definition 2.2]{phil1} (and throughout the present paper, and \cite[Problem 14G]{wil_top}, which also offers the alternative adverb `functionally' in place of `completely') it refers to a strictly stronger property than, say, on \cite[\S I.2, p.13]{countertop}, where it refers to the {\it $T_{2\frac 12}$ separation axiom}: distinct points have neighborhoods with disjoint closures.

    The strict strength ordering is not difficult to tease out of the said sources:
    \begin{itemize}
    \item On the one hand, if a continuous function $X\xrightarrow{f}[0,1]$ takes the value $0$ at $x$ and $1$ at $y$, then of course the two points have neighborhoods $f^{-1}[0,1/3)$ and $f^{-1}(1/3,1]$ with disjoint closures, so that
      \begin{equation}\label{eq:implt212}
        \text{functionally $T_2$}
        \Longrightarrow
        T_{2 \frac 12}.
      \end{equation}

    \item On the other hand, \cite[discussion on p.13]{countertop} notes that {\it $T_3$ spaces} \cite[Definition 14.1]{wil_top} are $T_{2\frac 12}$ (as the indexing suggests), while there are compactly generated $T_3$ spaces whereon real-valued functions fail to distinguish points: this is argued in \cite[Example 2.14]{phil1}, featuring the space also appearing in \cite[Problem 18G]{wil_top}, \cite[\S VII.3, Ex.3]{dug_top-12p} and originally in \cite[discussion following Lemma]{hew_2pbs}. 
    \end{itemize}
    
    
    The class of spaces relevant to the theory of commutative pro-$C^*$-algebras (functionally $T_2$ $\kappa$-spaces) is thus strictly narrower than that of $T_{2\frac 12}$ $\kappa$-spaces. Incidentally, said class is also strictly {\it broader} \cite[Example 2.13]{phil1} than that of $T_{3\frac 12}$ (i.e. {\it Tychonoff} \cite[Definition 14.8]{wil_top}) $\kappa$-spaces. The non-regular space in \cite[Problem 3J]{gj_rings} will do just as well: the usual topology on $\bR$, minimally extended so as to render $\{1/n,\ n\in\bZ_{>0}\}$ closed, is certainly functionally $T_2$ (because $\bR$ itself is), and compactly generated because \cite[Theorem 43.9 (b)]{wil_top} it is clearly {\it first-countable} \cite[Definition 10.3]{wil_top} (every point has a countable neighborhood base).

    Finally, \cite[Examples 2.13 and 2.14]{phil1} also jointly show that the class of interest is not comparable to that of $T_{3}$ $\kappa$-spaces. The functor diagram
    \begin{equation}\label{eq:howincluded}
      \begin{tikzpicture}[auto,baseline=(current  bounding  box.center)]
        \path[anchor=base] 
        (0,0) node (l) {$T_{3\frac 12,\kappa}$}
        +(4,1) node[draw,align=center,text width=2cm,inner sep=2pt,outer sep=2pt,rounded corners=0.2cm] (u) {functionally $T_2$ $\kappa$-spaces}
        +(4,-1) node (d) {$T_{3,\kappa}$}
        +(8,0) node (r) {$T_{2\frac 12,\kappa}$}
        ;
        \draw[right hook->] (l) to[bend left=6] node[pos=.5,auto] {$\scriptstyle \text{\cite[Example 2.13]{phil1}}$} (u);
        \draw[right hook->] (u) to[bend left=6] node[pos=.5,auto] {$\scriptstyle \text{\cite[Example 2.14]{phil1}}$} (r);
        \draw[right hook->] (l) to[bend right=6] node[pos=.5,auto,swap] {$\scriptstyle \text{\cite[Example 2.14]{phil1}}$} (d);
        \draw[right hook->] (d) to[bend right=6] node[pos=.5,auto,swap] {$\scriptstyle \text{\cite[\S II.78]{countertop}}$} (r);
        \draw[dashed,-] (u) to[bend right=0] node[pos=.5,auto,swap] {$\scriptstyle \text{\cite[2.13 \& 2.14]{phil1}}$} (d);
      \end{tikzpicture}
    \end{equation}
    summarizes the relationships between the various categories: the `$\kappa$' subscripts mean compact generation, and the dashed path indicates incomparability (i.e. neither category contains the other), and all inclusion functors are strict, as justified by the examples labeling the arrows.

    For the lower right-hand arrow compact generation, although not discussed in \cite[\S II.78]{countertop}, follows \cite[Theorem 43.9 (b)]{wil_top} again from first-countability.
  \end{enumerate}
\end{remarks}

\section{Complements on commutative pro-$C^*$-algebras}\label{se:compl-comm-pro}

Recall that
\begin{itemize}
\item A topological space is {\it pseudocompact} \cite[\S 1.4]{gj_rings} if every real-valued function on it is bounded.
\item A subspace $A\subseteq X$ is {\it relatively compact} \cite[\S 0.8]{ks_diff} if its closure is compact. 
\item And a subspace $A\subseteq X$ is {\it relatively pseudocompact} \cite[preceding Proposition 1.1]{schom_rel} if every continuous real-valued function on $X$ is bounded on $A$. 
\end{itemize}

\begin{proposition}\label{pr:whichbarr}
  A unital commutative pro-$C^*$-algebra is barreled if and only if it is of the form
  \begin{equation*}
    C(X)\text{ with the compact-open topology}
  \end{equation*}
  for some functionally $T_2$ space $X$ such that
  \begin{equation}\label{eq:droppseudo}
    \text{all relatively pseudocompact $A\subseteq X$ are relatively compact.}
  \end{equation}
\end{proposition}
\begin{proof}
  That barreled commutative unital pro-$C^*$-algebras are of the form $C(X)$ ($X$ functionally $T_2$) with uniform convergence on compact subsets of $X$ is, as already mentioned, observed in \cite[Remark 4.1]{inoue_loc} (see also \Cref{th:t312barrel} \Cref{item:t312barrel}, where the brief argument is restated). It is thus sufficient to consider pro-$C^*$-algebras $C(X)$, and argue that in that case being barreled is equivalent to \Cref{eq:droppseudo}. To see this, observe first that a locally convex TVS $V$ is barreled precisely when every {\it lower semicontinuous} \cite[\S 6.2]{koeth_tvs-1} seminorm
  \begin{equation*}
    V\ni v\xmapsto{\quad}\sup_{p\in \cP}p(v)\in \bR_{\ge 0}
  \end{equation*}
  is continuous for every pointwise-bounded family $\cP$ of continuous seminorms $p$: this follows from instance from
  \begin{itemize}
  \item the fact that all lower semicontinuous seminorms are of this form (\Cref{le:pbddfam});

  \item and the characterization \cite[\S 21.2 (2) and (3)]{koeth_tvs-1} of barreled TVSs as those whose lower semicontinuous seminorms are all automatically continuous. 
  \end{itemize}
  In the present setup it is enough to consider families of the form
  \begin{equation}\label{eq:pxfam}
    \{p_x:=|\text{evaluation at $x$}|\ |\ x\in A\}
  \end{equation}
  for subsets $A\subseteq X$, in which case pointwise boundedness simply means that $A$ is relatively pseudocompact. On the other hand, the supremum over \Cref{eq:pxfam} will be continuous only when dominated by the supremum over some compact subset of $X$, i.e. precisely when the closure $\overline{A}\subseteq X$ is compact.
\end{proof}

The proof of \Cref{pr:whichbarr} appeals to the following result; it is presumably well known, and \cite[Lemma VII.1.16, Proposition VII.4.2]{tak2} will serve the reader as close enough analogues (for {\it ordered} locally convex TVSs, with the proof easily adaptable). 

\begin{lemma}\label{le:pbddfam}
  Let $V$ be a real or complex locally convex TVS and $S(V)$ its associated family of continuous seminorms.
  \begin{enumerate}[(1)]

  \item\label{item:qissup} Every lower semicontinuous seminorm $q$ on $V$ is expressible as
    \begin{equation*}
      q(v) = \sup_{S(V)\ni p\le q}p(v).
    \end{equation*}

  \item\label{item:qissupbis} In particular, every lower semicontinuous seminorm is of the form
    \begin{equation*}
      q(v) = \sup_{p\in \cP}p(v)
    \end{equation*}
    for some pointwise-bounded subfamily $\cP\subseteq S(V)$.

  \end{enumerate}  
\end{lemma}
\begin{proof}
  \Cref{item:qissupbis} is clearly a consequence of \Cref{item:qissup}, so it is enough to address the latter. The {\it (closed) unit semiball} \cite[Definition 7.6]{trev_tvs}
  \begin{equation*}
    B_{q}:=\{v\in V\ |\ q(v)\le 1\}
  \end{equation*}
  is a barrel, as can easily be checked directly (\cite[Proposition 7.4]{trev_tvs} states this for {\it continuous} seminorms, but lower semicontinuity is all that is needed to ensure the closure of $B_q$).

  It follows from (one version of) the Hahn-Banach theorem (e.g. as a consequence of \cite[Corollary 1 to Proposition 18.2]{trev_tvs}) that $B_q$ is the intersection of the barrels $B$ containing it that are also origin neighborhoods. Every such $B$, though, is the closed unit semiball of a unique {\it continuous} \cite[Proposition 7.5]{trev_tvs} seminorm
  \begin{equation*}
    p_B(v):=\inf\left\{\lambda\ge 0\ |\ v\in\lambda B\right\}.
  \end{equation*}
  The resulting continuous seminorms $p_B$ are precisely those dominated by $q$, finishing the proof.
\end{proof}

\Cref{pr:whichbarr}, in particular, provides a rich supply of examples of barreled pro-$C^*$-algebras which are {\it not} $\sigma$-$C^*$. First, an immediate consequence of \Cref{pr:whichbarr}:

\begin{corollary}\label{cor:ismet}
  Every pro-$C^*$-algebra $C(X)$ with the compact-open topology and metrizable $X$ is barreled. 
\end{corollary}
\begin{proof}
  By \Cref{pr:whichbarr}, since for metrizable $X$ relatively pseudocompact subsets are relatively compact. Indeed, if $A\subseteq X$ is {\it not} compact then it contains \cite[Exercise 17G.3]{wil_top} some sequence $(x_n)$ with no convergent subsequence. The function
  \begin{equation*}
    x_n\xmapsto{\quad}n\in \bR
  \end{equation*}
  then extends continuously from the closed subset $\{x_n\}\subseteq X$ to all of $X$ \cite[Example 15.3 (c) and Theorem 15.8]{wil_top}, meaning that $\{x_n\}$ and hence $A$ is not relatively pseudocompact in $X$. 
\end{proof}

Consequently:

\begin{example}\label{ex:cq}
  Having topologized $\bQ\subset \bR$ as a subspace, $C(\bQ)$ is barreled by \Cref{cor:ismet} but not a $\sigma$-$C^*$-algebra \cite[Example 5.8]{phil1}.
\end{example}

In fact, the discussion producing \cite[Example 5.8]{phil1} extends to give a complete characterization of those metrizable spaces $X$ for which $C(X)$ is a $\sigma$-$C^*$-algebra. More generally:

\begin{proposition}\label{pr:1stcount}
  Let
  \begin{equation*}
    A\cong \left(C(X),\ \text{compact-open topology}\right)
  \end{equation*}
  be a commutative unital $\sigma$-$C^*$-algebra.

  \begin{enumerate}[(1)]
  \item\label{item:1stloc} If $x\in X$ has a countable system of neighborhoods then it has a compact neighborhood.

  \item\label{item:1stglob} In particular, if $X$ is first-countable then it is locally compact. 
  \end{enumerate}
\end{proposition}
\begin{proof}
  Part \Cref{item:1stglob} follows from \Cref{item:1stloc} (whose global counterpart it is), so we focus on the first claim. It, though, is essentially what the argument employed in \cite[Example 5.8]{phil1} proves: let
  \begin{equation*}
    x_n\xrightarrow[n]{\quad}x,\quad x_n\not\in K_n,\ \forall n\in \bZ_{>0}
  \end{equation*}
  be a convergent sequence, with $(K_n)$ denoting an increasing chain of compact subsets of $X$ defining the $\sigma$-$C^*$ topology. There is then \cite[Lemma 2.8]{phil1} a net of continuous functions converging to 0 on each $K_n$ but not on the compact set
  \begin{equation*}
    \left\{x_n,\ n\in \bZ_{>0}\right\}\cup \{x\},
  \end{equation*}
  contradicting the fact that the topology on $C(X)$ is that of uniform convergence on {\it all} compact subsets. 
\end{proof}

An immediate consequence (recalling \cite[Problem 17I]{wil_top} that a space is {\it $\sigma$-compact} if it is a countable union of compact subspaces):

\begin{corollary}\label{cor:whensigma1st}
  For a first-countable (in particular metrizable) functionally $T_2$ $X$ the algebra $C(X)$ with its compact-open topology is $\sigma$-$C^*$ if and only if $X$ is locally compact and $\sigma$-compact.  \qedhere
\end{corollary}

\begin{remarks}\label{res:howqworks}
  \begin{enumerate}[(1)]
  \item \Cref{cor:whensigma1st} makes it clear what the core issue driving \Cref{ex:cq} is: $\bQ$ is not locally compact. 

  \item There is a first part of \cite[Example 5.8]{phil1} arguing via \cite[Theorems 8.2 and 10.6]{gj_rings} that if $C(\bQ)$ is realized as a $\sigma$-$C^*$-algebra $C(X)$ then $X$ is homeomorphic to $\bQ$; this seems to me unnecessary, given that we already have a category equivalence \cite[Theorem 2.7 and Proposition 2.6]{phil1} between spaces-with-distinguished-sets and pro-$C^*$-algebras. 
  \end{enumerate}  
\end{remarks}

\cite[Remark 4.1]{inoue_loc} (as elaborated in \Cref{pr:whichbarr}) gives sufficient conditions for the distinguished family $\cF$ defining the pro-$C^*$ topology \Cref{eq:fc} on $\tensor[_{\cF}]{C(X)}{}$ to be full. This naturally prompts a more careful examination of how $\tensor[_{\cF}]{C(X)}{}$ {\it qua} pro-$C^*$-algebra reflects the point-set topology encoded by $\cF$.

As a concrete sample issue, \Cref{ex:taunetaub} notes in passing that under a very specific set of circumstances the bounded subsets of $\tensor[_{\cF}]{C(X)}{}$ are precisely the ``expected'' ones. The following simple remark records the fully general character of that phenomenon.

\begin{lemma}\label{le:allfsamebdd}
  A subset $B\subseteq \tensor[_{\cF}]{C(X)}{}$ of the pro-$C^*$-algebra \Cref{eq:fc} is abstractly bounded precisely when it is uniformly bounded on every compact subset of $X$:
  \begin{equation}\label{eq:bddeveryk}
    \sup_{f\in B,\ x\in K}|f(x)|<\infty,\quad\forall K\subseteq X\text{ compact}.
  \end{equation}
  In other words, the family of bounded subsets of $\tensor[_{\cF}]{C(X)}{}$ does not depend on $\cF$. 
\end{lemma}
\begin{proof}
  Abstract boundedness in $\tensor[_{\cF}]{C(X)}{}$ means \cite[Proposition 14.5]{trev_tvs} that \Cref{eq:bddeveryk} holds for all $K\in \cF$ (rather than all compact $K$ period), so the condition is clearly formally weaker than \Cref{eq:bddeveryk} itself; the interesting implication is thus the converse. Furthermore, restricting attention to a fixed but arbitrary $K\subseteq X$, there is no loss in assuming $X$ itself is compact (and Hausdorff) to begin with. The goal then becomes to show that if $B\subseteq \tensor[_{\cF}]{C(X)}{}$ is bounded then it is uniformly bounded.

  Suppose not. We can then find points $x_n\in X$, $n\in \bZ_{>0}$ such that
  \begin{equation*}
    \sup_{f\in B}|f(x_{n+1})|>\max\left(n,\ \sup_{f\in B}|f(x_{n})|\right)
    ,\quad
    \forall n\in \bZ_{>0}. 
  \end{equation*}
  These are clearly distinct, so that $\{x_n\}_n$ is an infinite set. It is then either closed, or not. In the former case we can always remove some cluster point from among the $x_n$ so as to break closure, so there is no harm in assuming that $\{x_n\}$ is {\it not} closed.

  Now, the topology-determination property \Cref{eq:topdet} of $\cF$ ensures the existence of some $F\in \cF$ for which $F\cap \{x_n\}$ is not closed. But then that set cannot be finite either, meaning that the family $B$ is not bounded on $F\in \cF$. This, in turn, violates the boundedness assumption and finishes the proof. 
\end{proof}

The extended discussion in \cite[\S 2]{phil1} (and partly recalled above in \Cref{res:whenbarr} \Cref{item:cautioncomplhaus}) makes it clear that the inadequacy of Tychonoff (i.e. $T_{3\frac 12}$, or {\it completely regular} and $T_1$ \cite[Definition 14.8]{wil_top}) spaces for the purpose of studying commutative pro-$C^*$-algebras is intimately linked to the other issue illustrated there \cite[Example 2.12]{phil1}, of considering not just plain spaces but rather spaces equipped with distinguished families of compact sets.

A slightly different take on the matter is that the category we will henceforth denote by $T_{2f,\kappa}$, of functionally $T_2$ $\kappa$-spaces (`f' for `functionally', `$\kappa$' for compact generation; top node of \Cref{eq:howincluded}), realizes a happy medium between two properties that are in tension:
\begin{itemize}
\item being Tychonoff, meaning \cite[Theorem 3.6]{gj_rings} that the space is equipped with {\it weakest (or initial)} topology \cite[Definition 8.9 and Historical Note to \S 8]{wil_top} induced by the family of continuous (real- or complex-valued) functions on the space;

\item and being compactly generated, meaning that the space is equipped with the {\it strongest (or initial)} topology \cite[Problem 9H and Historical Note to \S 9]{wil_top} induced by the compact subsets in the sense of \Cref{eq:topdet}.
\end{itemize}

Correspondingly, deficiencies in the two properties are mended in categorically dual ways, as we now recall.

\begin{notation}\label{not:idempmon}
  Write $\cat{Top}$ for the category of topological spaces.

  \begin{enumerate}[(1)]

  \item We denote by
    \begin{equation*}
      \cat{Top}
      \xrightarrow{\quad\tau\quad}
      \cat{Top}
    \end{equation*}
    the endofunctor associating to each space $X$ the Tychonoff space that $X$ maps into universally (\cite[Theorem 3.9]{gj_rings}, \cite[Problem 14H]{wil_top}).

    The image of $\tau$ is the full {\it reflective} \cite[Definition 3.5.2]{borc_hndbk-1} category $T_{3\frac 12}\subset\cat{Top}$ of Tychonoff spaces, and when regarded as a functor  
    \begin{equation*}
      \cat{Top}
      \xrightarrow{\quad\tau\quad}
      T_{3\frac 12}
    \end{equation*}
    (by a slight abuse of notation) $\tau$ is the {\it left adjoint} \cite[Definition 3.1.4]{borc_hndbk-1} of that inclusion of categories. Moreover, $\tau$ is the {\it idempotent monad} associated \cite[Corollary 4.2.4]{borc_hndbk-2} to the reflective subcategory $T_{3\frac 12}\subset \cat{Top}$.

    The picture pertaining to compact generation is dual. 

  \item\label{item:ktt} Denote by
    \begin{equation*}
      \cat{Top}
      \xrightarrow{\quad\kappa\quad}
      \cat{Top}
    \end{equation*}
    the endofunctor associating to each space $X$ the $\kappa$-space mapping universally into $X$ (akin to the {\it Kelleyfication} of \cite[\S VII.8]{mcl}, but with no assumption of Hausdorff-ness in this generality).

    The image of $\kappa$ is the full {\it coreflective} \cite[dual to Definition 3.5.2]{borc_hndbk-1} category $\cat{Top}_{\kappa}\subset\cat{Top}$ of $\kappa$-spaces, and when regarded as a functor  
    \begin{equation}\label{eq:kappattk}
      \cat{Top}
      \xrightarrow{\quad\kappa\quad}
      \cat{Top}_{\kappa}
    \end{equation}
    $\kappa$ is {\it right adjoint} \cite[p.98, following Definition 3.1.4]{borc_hndbk-1} to the inclusion. Moreover, $\kappa$ is the {\it idempotent comonad} associated via \cite[dual to Corollary 4.2.4]{borc_hndbk-2} to the coreflective subcategory $\cat{Top}_{\kappa}\subset \cat{Top}$.    
  \end{enumerate}
\end{notation}

\begin{remark}\label{re:kelleynini}
  In \cite[\S VII.8]{mcl} (as in \cite[\S 0]{dp_conv-2}) {\it Kelley} spaces are defined as $T_2$ $\kappa$-spaces, and the Kelleyfication mentioned in \Cref{not:idempmon} \Cref{item:ktt} is the right adjoint to the inclusion $T_{2,\kappa}\subset T_2$.

  If a space is already $T_2$ the functor $\kappa$ of \Cref{eq:kappattk} keeps it so, and hence \Cref{eq:kappattk} does indeed specialize to Kelleyfication when restricted to the category of Hausdorff spaces. On the other hand, the inclusion $T_2\subset \cat{Top}$ has \cite[3.1.6.h]{borc_hndbk-1} a {\it left} adjoint; the left and right adjunctions do not play well: the composite inclusion
  \begin{equation}\label{eq:t2kintop}
    \begin{tikzpicture}[auto,baseline=(current  bounding  box.center)]
      \path[anchor=base] 
      (0,0) node (l) {$T_{2,\kappa}$}
      +(2,.5) node (u) {$T_2$}
      +(2,-.5) node (d) {$\cat{Top}_{\kappa}$}
      +(4,0) node (r) {$\cat{Top}$}
      ;
      \draw[right hook->] (l) to[bend left=6] node[pos=.5,auto] {$\scriptstyle $} (u);
      \draw[right hook->] (u) to[bend left=6] node[pos=.5,auto] {$\scriptstyle $} (r);
      \draw[right hook->] (l) to[bend right=6] node[pos=.5,auto,swap] {$\scriptstyle $} (d);
      \draw[right hook->] (d) to[bend right=6] node[pos=.5,auto,swap] {$\scriptstyle $} (r);
    \end{tikzpicture}
  \end{equation}
  is neither a left nor a right adjoint (so there is no universal way of rendering an {\it arbitrary} space both compactly generated {\it and} $T_2$). Indeed, $T_{2,\kappa}$ is (co)complete \cite[\S VII.8, Proposition 2]{mcl}, but
  \begin{itemize}
  \item not preserving products \cite[paragraph preceding Definition 4.1]{steenr_conv}, the inclusion \Cref{eq:t2kintop} cannot \cite[\S V.5, Theorem 1]{mcl} be a right adjoint;

  \item and one can easily construct pairs of morphisms $X\to Y$ in $T_{2,\kappa}$ whose coequalizer in $\cat{Top}$ is not $T_2$, so that \Cref{eq:t2kintop} does not preserve coequalizers either (and hence is not a left adjoint, by \cite[\S V.5, dual to Theorem 1]{mcl}).
  \end{itemize}
\end{remark}

With this in place, and keeping to the convention of denoting (co)restrictions of a functor by the same symbol as the original functor, we have a diagram
\begin{equation}\label{eq:kt}
  \begin{tikzpicture}[auto,baseline=(current  bounding  box.center)]
    \path[anchor=base] 
    (0,0) node (l) {$T_{3\frac 12,\kappa}$}
    +(2,1) node (u) {$T_{2f,\kappa}$}
    +(2,-1) node (d) {$T_{3\frac 12}$}
    +(2,0) node (dash) {$\vdash$}
    ;
    \draw[right hook->] (l) to[bend left=6] node[pos=.5,auto] {$\scriptstyle $} (u);    
    \draw[right hook->] (l) to[bend right=6] node[pos=.5,auto,swap] {$\scriptstyle $} (d);
    \draw[->] (u) to[bend left=20] node[pos=.5,auto] {$\scriptstyle \tau$} (d);
    \draw[->] (d) to[bend left=20] node[pos=.5,auto] {$\scriptstyle \kappa$} (u);
  \end{tikzpicture}
\end{equation}
commutative up to the obvious natural isomorphisms (cf. \Cref{eq:howincluded}). The negative statement in \Cref{pr:ktadj} \Cref{item:ktnoteq} below is yet another  manifestation (albeit a somewhat roundabout one) of the fact that one cannot dispense with distinguished families of compact subsets when handling commutative pro-$C^*$-algebras. 

The symbol `$\vdash$' between the two vertical arrows indicates an adjunction, with the tail pointing towards the left adjoint (as in \Cref{pr:ktadj} \Cref{item:ktisadj} below); the notation is fairly standard (see \cite[Definition 19.3]{ahs}, for instance). 

\begin{proposition}\label{pr:ktadj}
  \begin{enumerate}[(1)]
  \item\label{item:ktisadj} The vertical functors in \Cref{eq:kt} form an adjunction, with $\kappa$ being the right adjoint.

  \item\label{item:ktfactiso} Writing
    \begin{equation*}
      \id\xrightarrow{\quad\eta\quad} \kappa\tau
      \quad\text{and}\quad
      \tau\kappa\xrightarrow{\quad\varepsilon\quad} \id
    \end{equation*}
    for the unit and counit of the adjunction of \Cref{item:ktisadj} respectively, the canonical factorizations \cite[Diagram 3.3]{borc_hndbk-1}
    \begin{equation}\label{eq:2fact}
      \begin{tikzpicture}[auto,baseline=(current  bounding  box.center)]
        \path[anchor=base] 
        (0,0) node (l) {$\tau$}
        +(2,.5) node (u) {$\tau\kappa\tau$}
        +(4,0) node (r) {$\tau$}
        ;
        \draw[->] (l) to[bend left=6] node[pos=.5,auto] {$\scriptstyle \tau\eta$} (u);
        \draw[->] (u) to[bend left=6] node[pos=.5,auto] {$\scriptstyle \varepsilon\tau$} (r);
        \draw[->] (l) to[bend right=6] node[pos=.5,auto,swap] {$\scriptstyle \id$} (r);
      \end{tikzpicture}
      \quad\text{and}\quad
      \begin{tikzpicture}[auto,baseline=(current  bounding  box.center)]
        \path[anchor=base] 
        (0,0) node (l) {$\kappa$}
        +(2,.5) node (u) {$\kappa\tau\kappa$}
        +(4,0) node (r) {$\kappa$}
        ;
        \draw[->] (l) to[bend left=6] node[pos=.5,auto] {$\scriptstyle \eta\kappa$} (u);
        \draw[->] (u) to[bend left=6] node[pos=.5,auto] {$\scriptstyle \kappa\varepsilon$} (r);
        \draw[->] (l) to[bend right=6] node[pos=.5,auto,swap] {$\scriptstyle \id$} (r);
      \end{tikzpicture}
    \end{equation}
    consist of natural isomorphisms.

  \item\label{item:ktcoreseq} The adjunction of \Cref{item:ktisadj} restricts to an equivalence between the essential images of $\kappa$ and $\tau$. 

  \item\label{item:ktisidemp} The monad $\kappa\tau$ and comonad $\tau\kappa$ attached \cite[Proposition 4.2.1]{borc_hndbk-2} to the adjunction are both idempotent.         
    
  \item\label{item:ktnoteq} $\kappa$ and $\tau$ are both faithful, but both fail to be either full or essentially surjective. In particular, $(\kappa,\tau)$ is not a pair of mutually-inverse equivalences. 
  \end{enumerate}
\end{proposition}
\begin{proof}
  \Cref{item:ktisadj} is a formal consequence of $\tau$ and $\kappa$ each being adjoint to the full inclusion of a category containing the domain of the other, as recalled in \Cref{not:idempmon}: for $Y\in T_{3\frac 12}$ and $X\in T_{2f,\kappa}$
  \begin{equation*}
    T_{3\frac 12}(\tau X,\ Y)
    \cong
    \cat{Top}(X,\ Y)
    \cong
    T_{2f,\kappa}(X,\ \kappa Y)
  \end{equation*}
  naturally. 
  
  To prove \Cref{item:ktfactiso}, observe that $\kappa$ leaves the underlying set of a topological space in place and strengthens the topology, whereas $\tau$ weakens it. It follows that on each space the two diagrams in \Cref{eq:2fact} are factorizations of identity maps through a weaker (left hand) and stronger (right hand) topology, and hence homeomorphisms. 
  
  \Cref{item:ktcoreseq} and \Cref{item:ktisidemp} follow immediately from \Cref{item:ktfactiso}. 
  
  As for \Cref{item:ktnoteq}, faithfulness is obvious, since the two functors do not in fact alter morphisms at all. It follows from \Cref{item:ktfactiso} that the objects
  \begin{equation*}
    X\in\text{essential image of }\kappa\subseteq T_{2f,\kappa}
  \end{equation*}
  are precisely those for which $X\xrightarrow{\eta}\kappa\tau X$ is an isomorphism, and similarly (or rather dually) for $\tau$. The full faithfulness of $\tau$ (i.e. its fullness, since faithfulness we already have) and the essential surjectivity of $\kappa$ are thus \cite[Proposition 3.4.1]{borc_hndbk-1} both equivalent to $\eta$ being an isomorphism; the claim is that it isn't. Dually, we also have to prove that the counit $\tau\kappa\xrightarrow{\varepsilon}\id$ is not a natural isomorphism; \Cref{ex:kt1,ex:tk1} show this, respectively.
\end{proof}

\begin{example}\label{ex:kt1}
  Consider the space $E$ of \cite[Problem 3J]{gj_rings} (mentioned also in \Cref{res:whenbarr} \Cref{item:cautioncomplhaus}, where we argue that the space does belong to $T_{2f,\kappa}$): the real line $\bR$, with its usual topology minimally extended so as to make the set $\left\{\frac 1n\right\}_{n\in\bZ_{>0}}$ closed. 

  I claim that the canonical map $E\xrightarrow{\quad} \kappa\tau E$ is not a homeomorphism. To see this, note that $\tau E$ is nothing but $\bR$ by \cite[3J 3.]{gj_rings}, so in particular $[0,1]\subset E$ will be compact in both $\tau E$ and $\kappa\tau E$. It is not \cite[3J 4.]{gj_rings}, though, (even countably) compact in $E$ itself.
\end{example}

\begin{example}\label{ex:tk1}
  Anything resembling \cite[Example 2.12]{phil1} will do: consider some infinite set $D$ and set
  \begin{equation*}
    X:=D\sqcup\{p\}\subset \beta D,\quad p\in\beta D\setminus D
  \end{equation*}
  with the subspace topology, where $\beta D$ denotes the {\it Stone-\v{C}ech compactification} \cite[\S 6]{gj_rings} of $D$ equipped with its discrete topology. Plainly, $X$ is $T_{3\frac 12}$ because \cite[Theorem 3.14]{gj_rings} it is by definition subspace of a compact Hausdorff space.
  
  The claim is that the canonical map $\tau\kappa X\xrightarrow{\quad}X$ is not a homeomorphism. Much more is true: $\kappa X$ (and hence $\tau\kappa X$) is discrete. This, in turn, follows from the observation that a subset $Y\subseteq X$ is compact precisely when it is finite (see also \cite[Example 12.5]{mich_a0}). Indeed, an infinite $D'\subseteq D$ can be partitioned into two infinite subsets $D'_i$, $i=0,1$. At least one of them has \cite[Theorem 6.5]{gj_rings} a cluster point
  \begin{equation*}
    q\ne p,\quad q\in \beta D,
  \end{equation*}
  so that $D'\sqcup\{p\}$ cannot be compact. 
\end{example}

As indicated above, mending the pathology noted in \Cref{pr:ktadj} \Cref{item:ktnoteq} will require equipping spaces with families of compact subsets, as in \cite[Definition 2.5]{phil1} (so that theme recurs). We need some terminology and notation.

\begin{definition}\label{def:funcdist}
  Let $X$ be a topological space. A family $\cF$ of compact subspaces of $X$ is {\it functionally distinguished} if it
  \begin{itemize}
  \item contains the singletons;

  \item is closed under taking compact subsets;

  \item is closed under forming finite (or equivalently, binary) unions;
    
  \item and determines the (real- or, equivalently, complex-valued) continuous functions on $X$ in the sense that
    \begin{equation}\label{eq:funcdet}
      X\xrightarrow{\ f\ }\bR\text{ is continuous }
      \iff
      f|_F\text{ is continuous },\ \forall F\in \cF.
    \end{equation}
  \end{itemize}
  For any {\it concrete} category $\cC$ of topological spaces (i.e. category equipped with a distinguished faithful functor to $\cat{Top}$; what \cite[Definition 5.1]{ahs} would call {\it $\cat{Top}$-concrete}) we write
  \begin{itemize}
  \item $\tensor[_{\Phi}]{\cC}{}$ for the category consisting of objects in $\cC$ equipped with distinguished families of compact subsets in the sense of \Cref{res:whenbarr} \Cref{item:fcx}, with morphisms required to preserve members of the distinguished families.

  \item and similarly, $\tensor[_{f\Phi}]{\cC}{}$ for the analogue with {\it functionally} distinguished families instead. 
  \end{itemize}
\end{definition}

\begin{remark}\label{re:whatfdistmeans}
  Being {\it functionally} distinguished, in other words, differs from being distinguished only in the fourth bullet point of \Cref{def:funcdist} (cf. \cite[Definition 2.5]{phil1}, as recalled in \Cref{res:whenbarr} \Cref{item:fcx}): distinguished families determine the topology, whereas functionally distinguished ones only determine the continuous functions on a space (and hence capture {\it less} information). 
\end{remark}

The functors $\kappa$ and $\tau$ of \Cref{eq:kt} transfer over to the distinguished setting:

\begin{lemma}\label{le:ktextend}
  The vertical functors in \Cref{eq:kt} extend to an adjunction
  \begin{equation*}
    \begin{tikzpicture}[auto,baseline=(current  bounding  box.center)]
      \path[anchor=base] 
      (0,0) node (l) {$\tensor[_{\Phi}]{T}{_{2f,\kappa}}$}
      +(4,0) node (r) {$\tensor[_{f\Phi}]{T}{_{3\frac 12}}$}
      +(2,0) node (m) {$\bot$}
      ;
      \draw[->] (l) to[bend left=20] node[pos=.5,auto] {$\scriptstyle \tau$} (r);
      \draw[->] (r) to[bend left=20] node[pos=.5,auto] {$\scriptstyle \kappa$} (l);
    \end{tikzpicture}
  \end{equation*}
  by acting identically on families of compact subsets. 
\end{lemma}
\begin{proof}
  Once we argue that the functors do indeed extend as claimed, the adjunction will be no more difficult to prove than in \Cref{pr:ktadj}, so we will forego that portion of the argument.

  To prove the extension claim, note first that the original $\kappa$ and $\tau$ of \Cref{eq:kt} preserve embeddings of compact Hausdorff subspaces (this is not so for the functors originally defined on all of $\cat{Top}$ in \Cref{not:idempmon}!), so at the very least, it does make sense to operate identically on subsets of compact spaces. The crux of the matter is to argue that
  \begin{enumerate}[(I)]
  \item\label{item:d2fd} if the family $\cF\subseteq 2^X$ of compact subsets of some $X\in T_{2f,\kappa}$ is distinguished then it becomes functionally distinguished in $\tau X$;

  \item\label{item:fd2d} and vice versa for $\kappa$. 
  \end{enumerate}

  Claim \Cref{item:d2fd} is just an unpacking of the various universality properties: $\cF$, being distinguished, $X$ is the {\it directed colimit} \cite[Example 11.28 (4)]{ahs} in $\cat{Top}$ of the inclusions $F\subseteq X$, $F\in\cF$, so that a function $X\to Y$ to {\it any} topological space (such as $Y=\bR$, say) is continuous precisely when its restriction to every $F\in \cF$ is.

  As for \Cref{item:fd2d}: given an object $(X,\cF)$ of $\tensor[_{f\Phi}]{T}{_{3\frac 12}}$, the closed subsets of $\kappa X$ are those subsets of $X$ whose intersection with every compact subset of $X$ is closed; what we have to argue is that under the present assumption that $\cF$ is functionally distinguished, it is enough to intersect with every $F\in \cF$.

  That claim, in turn, follows from the fact that the original topology on every compact $K\subseteq X$ is the weak topology induced by $C(X)$ (since $K$ itself is compact Hausdorff and hence $T_{3\frac 12}$ \cite[Theorem 3.14]{gj_rings}, and continuous functions extend \cite[\S 3.11 (c)]{gj_rings} from $K$ to $X$), so that, $\cF$ being functionally distinguished, $K$ is also equipped with the strong topology generated by the intersections $F\cap K$ for $F\in \cF$.
\end{proof}

In the following statement, note that the category $T_{3\frac 12,\kappa}$ of Tychonoff $\kappa$-spaces embeds naturally into both $\tensor[_{\Phi}]{T}{_{2f,\kappa}}$ and $\tensor[_{f\Phi}]{T}{_{3\frac 12}}$ by simply equipping every such space (Tychonoff {\it and} compactly generated) with its full family of compact subspaces. More generally, for a Tychonoff $\kappa$-space a family $\cF$ is distinguished {\it precisely} when it is functionally distinguished, so we have an identification
\begin{equation*}
  \tensor[_{\Phi}]{T}{_{3\frac 12,\kappa}}
  =
  \tensor[_{f\Phi}]{T}{_{3\frac 12,\kappa}}
\end{equation*}

\begin{theorem}\label{th:ktcorr}
  In the diagram
  \begin{equation}\label{eq:ktdisteq}
    \begin{tikzpicture}[auto,baseline=(current  bounding  box.center)]
      \path[anchor=base] 
      (0,0) node (l) {$\tensor[_\Phi]{T}{_{3\frac 12,\kappa}}$}
      +(-2,0) node (ll) {$T_{3\frac 12,\kappa}$}
      +(2,1) node (u) {$\tensor[_{\Phi}]{T}{_{2f,\kappa}}$}
      +(2,-1) node (d) {$\tensor[_{f\Phi}]{T}{_{3\frac 12}}$}
      +(2,0) node (eq) {$\simeq$}
      ;
      \draw[right hook->] (ll) to[bend left=0] node[pos=.5,auto] {$\scriptstyle $} (l);
      \draw[right hook->] (l) to[bend left=6] node[pos=.5,auto] {$\scriptstyle $} (u);
      \draw[right hook->] (l) to[bend right=6] node[pos=.5,auto,swap] {$\scriptstyle $} (d);
      \draw[->] (u) to[bend left=20] node[pos=.5,auto] {$\scriptstyle \tau$} (d);
      \draw[->] (d) to[bend left=20] node[pos=.5,auto] {$\scriptstyle \kappa$} (u);
    \end{tikzpicture}
  \end{equation}
  commutative up to the obvious natural isomorphisms, the vertical functors are mutually inverse equivalences with $\kappa$ right adjoint to $\tau$. 
\end{theorem}
\begin{proof}
  The adjunction claim is part of \Cref{le:ktextend}, and the fact that the unit and counit are isomorphisms is virtually a tautology.

  Consider, for instance, the unit $\id\xrightarrow{\eta}\kappa\tau$. One starts with a functionally $T_2$ $\kappa$-space $X$ whose topology is defined by a family $\cF$ of compact sets, applies $\tau$ so as to weaken that topology into a Tychonoff one, equipped with {\it the same} (now only {\it functionally} distinguished) family $\cF$, and then applies $\kappa$ to strengthen {\it that} topology back into the one generated by $\cF$. Plainly, we have the selfsame topology on $X$ that we began with. 
\end{proof}

\begin{remark}
  \Cref{th:ktcorr} is not surprising, in light of \Cref{re:whatfdistmeans}. As the latter suggests, there is a reciprocity in passing back and forth between $\tensor[_{f\Phi}]{T}{_{3\frac 12}}$ and $\tensor[_{\Phi}]{T}{_{2f,\kappa}}$: the stronger requirement that the topology be determined by continuous functions, which is effectively what the $T_{3\frac 12}$ property means \cite[Theorem 3.6]{gj_rings}, is compensated for by equipping said spaces with only {\it functionally} distinguished families on the $\tensor[_{f\Phi}]{T}{_{3\frac 12}}$ side of the equivalence. 
\end{remark}

Transporting \cite[Theorem 2.7]{phil1} over to Tychonoff spaces along the equivalence of \Cref{th:ktcorr}, we have the following alternative description of unital commutative pro-$C^*$-algebras. 

\begin{corollary}\label{cor:alsot312}
  The functor
  \begin{equation*}
    \tensor[_{f\Phi}]{T}{_{3\frac 12}}
    \ni (X,\cF)
    \xmapsto{\quad}
    \tensor[_\cF]{C(X)}{}
  \end{equation*}
  is a contravariant equivalence between the category of Tychonoff spaces equipped with functionally distinguished families of compact sets and that of unital commutative pro-$C^*$-algebras.  \qedhere
\end{corollary}

There is also a Tychonoff version of \Cref{pr:whichbarr}:

\begin{theorem}\label{th:t312barrel}
  \begin{enumerate}[(1)]

  \item\label{item:t312barrel} A unital commutative pro-$C^*$-algebra is barreled if and only if it is of the form
    \begin{equation*}
      C(X)\text{ with the compact-open topology}
    \end{equation*}
    for some $T_{3\frac 12}$ space $X$ such that
    \begin{equation*}
      \text{all relatively pseudocompact $A\subseteq X$ are relatively compact.}
    \end{equation*}

  \item\label{item:samecpct} If $X\in T_{2f,\kappa}$ is a functionally Hausdorff $\kappa$-space whose relatively pseudocompact subsets are all relatively compact, then the compact subspaces of $\tau X\in T_{3\frac 12}$ are then precisely those of the original topology.

  \end{enumerate}
\end{theorem}
\begin{proof}
  We address the claims in reverse order.

  \begin{enumerate}[]

  \item {\bf \Cref{item:t312barrel} $\Longrightarrow$ \Cref{item:samecpct}:} The spaces $X$ as in the statement are, by \Cref{pr:whichbarr}, precisely those for which $C(X)$ is barreled when equipped with the compact-open topology. By \Cref{cor:alsot312} this is also $\tensor[_{\cF}]{C(X)}{}$ for the family $\cF\subset 2^X$ of subsets compact in the original topology on $X$, and that family must consist of {\it all} compact subsets of $\tau X$ by part \Cref{item:t312barrel} (which we are here assuming).

  \item {\bf \Cref{item:t312barrel}: } The proof of \Cref{pr:whichbarr} simply carries over {\it mutatis mutandis}, once we retrieve \cite[Remark 4.1]{inoue_loc} in the present context: if $\tensor[_{\cF}]{C(X)}{}$ is barreled then the family $\cF$ must be full because the supremum over every compact set is a lower semicontinuous seminorm.
  \end{enumerate}
  This finishes the proof. 
\end{proof}

\begin{remarks}\label{res:needextraconstr}
  \begin{enumerate}[(1)]
  \item Absent some additional constraint, such as the barreled-ness implicit (via \Cref{pr:whichbarr}) in \Cref{th:t312barrel} \Cref{item:samecpct}, one cannot expect the equivalence $\tau$ of \Cref{eq:ktdisteq} to turn full families of compact subsets into {\it full} families again. The two examples of non-Tychonoff $T_{2f,\kappa}$-spaces mentioned in \Cref{res:whenbarr} \Cref{item:cautioncomplhaus}, for instance, exhibit this pathology.

    For the space $E$ of \cite[Problem 3J]{gj_rings}, say, the resulting functionally distinguished family of compact sets on
    \begin{equation*}
      \tau E\cong \bR\text{ with the usual topology}
    \end{equation*}
    is that of compact subsets whose intersection with $\left\{\frac 1n\ |\ n\in\bZ_{>0}\right\}$ is finite. That family is {\it not} distinguished, so the example falls outside the common horizontal core of \Cref{eq:ktdisteq} and displays the non-trivial behavior of the vertical equivalence in that diagram.

  \item\label{item:inoueagain} Such examples also again show (as did \Cref{res:whenbarr} \Cref{item:fcx}) that for a functionally $T_2$ $\kappa$-space $X$, $\tensor[_{\cF}]{C(X)}{}$ being barreled is strictly stronger than $\cF$ being full.
  \end{enumerate}
\end{remarks}

\addcontentsline{toc}{section}{References}

\begin{thebibliography}{10}

\bibitem{ahs}
Ji{\v{r}}{\'{\i}} Ad{\'a}mek, Horst Herrlich, and George~E. Strecker.
\newblock {\em Abstract and concrete categories. {The} joy of cats}.
\newblock New York etc.: John Wiley \& Sons, Inc., 1990.

\bibitem{apt}
Charles~A. Akemann, Gert~K. Pedersen, and Jun Tomiyama.
\newblock Multipliers of {$C\sp*$}-algebras.
\newblock {\em J. Functional Analysis}, 13:277--301, 1973.

\bibitem{bog_meas}
V.~I. Bogachev.
\newblock {\em Measure theory. {Vol}. {I} and {II}}.
\newblock Berlin: Springer, 2007.

\bibitem{borc_hndbk-1}
Francis Borceux.
\newblock {\em Handbook of categorical algebra. 1}, volume~50 of {\em
  Encyclopedia of Mathematics and its Applications}.
\newblock Cambridge University Press, Cambridge, 1994.
\newblock Basic category theory.

\bibitem{borc_hndbk-2}
Francis Borceux.
\newblock {\em Handbook of categorical algebra. 2}, volume~51 of {\em
  Encyclopedia of Mathematics and its Applications}.
\newblock Cambridge University Press, Cambridge, 1994.
\newblock Categories and structures.

\bibitem{bourb_tvs}
N.~Bourbaki.
\newblock {\em Topological vector spaces. {C}hapters 1--5}.
\newblock Elements of Mathematics (Berlin). Springer-Verlag, Berlin, 1987.
\newblock Translated from the French by H. G. Eggleston and S. Madan.

\bibitem{ct_ncclass_xv}
Alexandru Chirvasitu and Mariusz Tobolski.
\newblock Non-commutative classifying spaces of groups via quasi-topologies and
  pro-$c^*$-algebras, 2023.

\bibitem{conw_strict-1}
John~B. Conway.
\newblock The strict topology and compactness in the space of measures.
\newblock {\em Bull. Amer. Math. Soc.}, 72:75--78, 1966.

\bibitem{dixw}
Jacques Dixmier.
\newblock {\em von {N}eumann algebras}, volume~27 of {\em North-Holland
  Mathematical Library}.
\newblock North-Holland Publishing Co., Amsterdam-New York, 1981.
\newblock With a preface by E. C. Lance, Translated from the second French
  edition by F. Jellett.

\bibitem{dorr_loc}
J.~R. Dorroh.
\newblock The localization of the strict topology via bounded sets.
\newblock {\em Proc. Amer. Math. Soc.}, 20:413--414, 1969.

\bibitem{dp_conv-2}
Eduardo~J. Dubuc and Horacio Porta.
\newblock Convenient categories of topological algebras, and their duality
  theory.
\newblock {\em J. Pure Appl. Algebra}, 1:281--316, 1971.

\bibitem{dug_top-12p}
James Dugundji.
\newblock {\em Topology.}
\newblock Allyn and Bacon, Inc., Boston, Mass.-London-Sydney,,, 1978.
\newblock Reprinting of the 1966 original.

\bibitem{gj_rings}
Leonard Gillman and Meyer Jerison.
\newblock {\em Rings of continuous functions}.
\newblock Graduate Texts in Mathematics, No. 43. Springer-Verlag, New
  York-Heidelberg, 1976.
\newblock Reprint of the 1960 edition.

\bibitem{hew_2pbs}
Edwin Hewitt.
\newblock On two problems of {Urysohn}.
\newblock {\em Ann. Math. (2)}, 47:503--509, 1946.

\bibitem{inoue_loc}
Atsushi Inoue.
\newblock Locally {$C^{\ast} $}-algebra.
\newblock {\em Mem. Fac. Sci. Kyushu Univ. Ser. A}, 25:197--235, 1971.

\bibitem{ks_diff}
Victor Khatskevich and David Shoiykhet.
\newblock {\em Differentiable operators and nonlinear equations. {Translated}
  from the {Russian} by {Mircea} {Martin}}, volume~66 of {\em Oper. Theory:
  Adv. Appl.}
\newblock Basel: Birkh{\"a}user Verlag, 1994.

\bibitem{koeth_tvs-1}
Gottfried K\"{o}the.
\newblock {\em Topological vector spaces. {I}}.
\newblock Die Grundlehren der mathematischen Wissenschaften, Band 159.
  Springer-Verlag New York, Inc., New York, 1969.
\newblock Translated from the German by D. J. H. Garling.

\bibitem{mcl}
Saunders Mac~Lane.
\newblock {\em Categories for the working mathematician}, volume~5 of {\em
  Graduate Texts in Mathematics}.
\newblock Springer-Verlag, New York, second edition, 1998.

\bibitem{mich_a0}
Ernest Michael.
\newblock {{\(\aleph_ 0\)}}-spaces.
\newblock {\em J. Math. Mech.}, 15:983--1002, 1966.

\bibitem{mich_loc}
Ernest~A. Michael.
\newblock {\em Locally multiplicatively-convex topological algebras}, volume~11
  of {\em Mem. Am. Math. Soc.}
\newblock Providence, RI: American Mathematical Society (AMS), 1952.

\bibitem{mnk}
James~R. Munkres.
\newblock {\em Topology}.
\newblock Prentice Hall, Inc., Upper Saddle River, NJ, 2000.
\newblock Second edition of [ MR0464128].

\bibitem{phil1}
N.~Christopher Phillips.
\newblock Inverse limits of {$C^*$}-algebras.
\newblock {\em J. Operator Theory}, 19(1):159--195, 1988.

\bibitem{rr_tvs}
A.~P. Robertson and Wendy Robertson.
\newblock {\em Topological vector spaces}, volume~53 of {\em Cambridge Tracts
  in Mathematics}.
\newblock Cambridge University Press, Cambridge-New York, second edition, 1980.

\bibitem{schae_tvs}
H.~H. Schaefer and M.~P. Wolff.
\newblock {\em Topological vector spaces}, volume~3 of {\em Graduate Texts in
  Mathematics}.
\newblock Springer-Verlag, New York, second edition, 1999.

\bibitem{schom_rel}
John~J. Schommer.
\newblock Relatively realcompact sets and nearly pseudocompact spaces.
\newblock {\em Comment. Math. Univ. Carolin.}, 34(2):375--382, 1993.

\bibitem{countertop}
Lynn~Arthur Steen and J.~Arthur Seebach, Jr.
\newblock {\em Counterexamples in topology}.
\newblock Dover Publications, Inc., Mineola, NY, 1995.
\newblock Reprint of the second (1978) edition.

\bibitem{steenr_conv}
N.~E. Steenrod.
\newblock A convenient category of topological spaces.
\newblock {\em Mich. Math. J.}, 14:133--152, 1967.

\bibitem{tak1}
M.~Takesaki.
\newblock {\em Theory of operator algebras. {I}}, volume 124 of {\em
  Encyclopaedia of Mathematical Sciences}.
\newblock Springer-Verlag, Berlin, 2002.
\newblock Reprint of the first (1979) edition, Operator Algebras and
  Non-commutative Geometry, 5.

\bibitem{tak2}
M.~Takesaki.
\newblock {\em Theory of operator algebras. {II}}, volume 125 of {\em
  Encyclopaedia of Mathematical Sciences}.
\newblock Springer-Verlag, Berlin, 2003.
\newblock Operator Algebras and Non-commutative Geometry, 6.

\bibitem{tay_strict}
Donald~Curtis Taylor.
\newblock The strict topology for double centralizer algebras.
\newblock {\em Trans. Amer. Math. Soc.}, 150:633--643, 1970.

\bibitem{stacks-project}
{The Stacks Project Authors}.
\newblock \textit{Stacks Project}.
\newblock \url{https://stacks.math.columbia.edu}, 2018.

\bibitem{trev_tvs}
Fran\c{c}ois Tr\`eves.
\newblock {\em Topological vector spaces, distributions and kernels}.
\newblock Dover Publications, Inc., Mineola, NY, 2006.
\newblock Unabridged republication of the 1967 original.

\bibitem{wo}
N.~E. Wegge-Olsen.
\newblock {\em {$K$}-theory and {$C^*$}-algebras}.
\newblock Oxford Science Publications. The Clarendon Press, Oxford University
  Press, New York, 1993.
\newblock A friendly approach.

\bibitem{wil_top}
Stephen Willard.
\newblock {\em General topology}.
\newblock Dover Publications, Inc., Mineola, NY, 2004.
\newblock Reprint of the 1970 original [Addison-Wesley, Reading, MA;
  MR0264581].

\end{thebibliography}

\def\polhk#1{\setbox0=\hbox{#1}{\ooalign{\hidewidth
  \lower1.5ex\hbox{`}\hidewidth\crcr\unhbox0}}}
  \def\polhk#1{\setbox0=\hbox{#1}{\ooalign{\hidewidth
  \lower1.5ex\hbox{`}\hidewidth\crcr\unhbox0}}}
  \def\polhk#1{\setbox0=\hbox{#1}{\ooalign{\hidewidth
  \lower1.5ex\hbox{`}\hidewidth\crcr\unhbox0}}}
  \def\polhk#1{\setbox0=\hbox{#1}{\ooalign{\hidewidth
  \lower1.5ex\hbox{`}\hidewidth\crcr\unhbox0}}}
  \def\polhk#1{\setbox0=\hbox{#1}{\ooalign{\hidewidth
  \lower1.5ex\hbox{`}\hidewidth\crcr\unhbox0}}}

\Addresses

\end{document}